\documentclass{amsart}
\usepackage{amssymb}
\usepackage{amsmath}
\usepackage{amsthm}
\usepackage{paralist}
\usepackage[misc]{ifsym}
\usepackage{epsfig} 
\usepackage{epstopdf} %This is to transfer .eps figure to .pdf figure; please compile your article using PDFLaTex or PDFTeXify.
\usepackage[hidelinks]{hyperref}
\allowdisplaybreaks
\usepackage[all,cmtip]{xy}
\usepackage{xcolor}
\usepackage{mathrsfs}
\usepackage{subfigure}

\newcommand{\goodgap}{\hspace{\subfigtopskip}\hspace{\subfigbottomskip}}

\newtheorem{theorem}{Theorem}[section]
\newtheorem{corollary}[theorem]{Corollary}

\newtheorem{lemma}[theorem]{Lemma}
\newtheorem{proposition}[theorem]{Proposition}

\theoremstyle{definition}
\newtheorem{definition}[theorem]{Definition}
\newtheorem{remark}[theorem]{Remark}
\newtheorem{example}[theorem]{Example}

\newcommand{\lra}{\longrightarrow}
\newcommand{\rr}{\mathbb{R}}
\newcommand{\ff}{\mathbb F}
\newcommand{\frakA}{\mathfrak A}
\newcommand{\frakg}{\mathfrak g}
\newcommand{\frakS}{\mathfrak S}
\newcommand{\bF}{\mathbf F}
\newcommand{\Lie}{\operatorname{Lie}}
\newcommand{\Ad}{\operatorname{Ad}}
\newcommand{\Hor}{\operatorname{Hor}}
\newcommand{\calA}{\mathcal A}
\newcommand{\scrA}{\mathscr A}
\newcommand{\calH}{\mathcal H}
\newcommand{\calL}{\mathcal L}
\newcommand{\calV}{\mathcal V}
\newcommand{\calX}{\mathcal X}
\newcommand{\calY}{\mathcal Y}
\newcommand{\subs}{\subset}
\newcommand{\pr}{\operatorname{pr}}

\newcommand{\norm}[1]{\lVert#1\rVert}
\newcommand{\Span}{\operatorname{span}}
\newcommand{\diag}{\operatorname{diag}}

\newcommand{\ti}[1]{\widetilde{{#1}}}
\newcommand{\baseRing}[1]{\ensuremath{\mathbb{#1}}}
\newcommand{\R}{\baseRing{R}}

\newcommand{\jdef}[1]{\emph{#1}}

\newcommand{\del}{\partial}
\newcommand{\CD}{\ensuremath{{\calA_d}}}

\newcommand{\HLc}[1]{\ensuremath{{h^{{#1}}}}}
\newcommand{\SG}{\ensuremath{G}}
\newcommand{\RS}{\ensuremath{\ti{\SG}\times (Q/\SG)}}

\usepackage[
backend=biber,
style=alphabetic,
sorting=nyt,
uniquename=false,
giveninits=true,
doi=false,url=false
]{biblatex}
\renewbibmacro{in:}{} % Saca el "In:"
\begin{document}

\title[Remarks on structures and preservation]{Remarks on structures and preservation in forced discrete mechanical systems of Routh type}

\author{Mat\'ias I. Caruso}
\address{\textnormal{(M. I. Caruso)} Instituto Balseiro \\ Universidad Nacional de Cuyo - C.N.E.A. \hfill \break Av. Bustillo 9500 \\ San Carlos de Bariloche \\ R8402AGP \\ Argentina \hfill \break
	\indent CONICET}
\email{matias.caruso@ib.edu.ar}

\author{Javier Fern\'andez}
\address{\textnormal{(J. Fern\'andez)} Instituto Balseiro \\ Universidad Nacional de Cuyo - C.N.E.A. \hfill \break Av. Bustillo 9500 \\ San Carlos de Bariloche \\ R8402AGP \\ Argentina}
\email{jfernand@ib.edu.ar}

\author{Cora Tori}
\address{\textnormal{(C. Tori)} Depto. de Ciencias B\'asicas \\ Facultad de Ingenier\'ia \\ Universidad Nacional de La Plata.
	Calle 116 entre 47 y 48 \\ La Plata \\ Buenos Aires \\ 1900 \\ Argentina \hfill \break
	\indent Centro de Matem\'atica de La Plata (CMaLP)}
\email{cora.tori@ing.unlp.edu.ar}

\author{Marcela Zuccalli}
\address{\textnormal{(M. Zuccalli)} Depto. de Matem\'atica \\ Facultad de Ciencias Exactas \\ Universidad Nacional de La Plata.
	Calles 50 y 115 \\ La Plata \\ Buenos Aires \\ 1900 \\ Argentina \hfill \break
	\indent Centro de Matem\'atica de La Plata (CMaLP)}
\email{marce@mate.unlp.edu.ar}

\subjclass{Primary: 37J06, 70H33; Secondary: 70G75.}
\keywords{Geometric mechanics, forced discrete mechanical systems, discrete Routh reduction.}

\begin{abstract}
We study a type of forced discrete mechanical system $(Q,L_d,f_d)$
---that we name of Routh type--- whose (discrete) time-flow preserves
a symplectic structure on $Q\times Q$. That structure arises as the
pullback via the forced discrete Legendre transform of the canonical
symplectic structure on $T^*Q$ modified by a ``magnetic term''. One
example of this type of system is provided by the Lagrangian reduction
of a symmetric (unforced) discrete mechanical system in the Routh
style. In this particular case, we do not reduce by the full symmetry
group but, rather, by an appropriate isotropy subgroup. In this
context, the preserved symplectic structure can be alternatively seen
as the Marsden-Weinstein reduction of the canonical symplectic
structure $\omega_{L_d}$ on $Q\times Q$.
\end{abstract}

\maketitle

\section{Introduction}

Discrete mechanics can be used to construct computational algorithms
that are, in turn, used to model the evolution of (continuous time)
mechanical systems
(see~\cite{M-West}). Part
of the success of this application has been ascribed to the fact that,
in many instances, discrete mechanical systems have structural
properties that mirror those of their continuous counterparts. For
example, the time evolution of a Lagrangian mechanical system on a
configuration manifold $Q$ with Lagrangian function
$L : TQ \lra \mathbb{R}$ preserves a certain symplectic form $\omega_L$ on $TQ$, which can be seen as the pullback via the Legendre transform of the canonical symplectic form $\omega_Q$ on $T^*Q$. On the
discrete side, a discrete mechanical system on $Q$ with discrete
Lagrangian function $L_d:Q\times Q\lra \R$ preserves a
symplectic structure $\omega_{L_d}$ on $Q\times Q$ that can be seen as
the pullback via the discrete Legendre transform of
$\omega_Q$. Also, key techniques and points of view in the study
of both types of systems have been developed that run parallel to each
other; one such example is the theory of Lagrangian or variational
reduction.

The description of real world mechanical systems usually involves
external forces (friction, dissipation, etc.) that, somehow, preclude
the existence of certain conserved magnitudes. For example, in most
circumstances, the symplectic structures described above are not
preserved by the flow of a forced system. Still, in some cases, there
are alternative symplectic structures that are preserved. One such
instance occurs when the force of the system is an exact differential form. In that case, a modified Lagrangian can be defined, whose associated symplectic form is now preserved. One of the goals of this paper is to study another family of forces that exhibit the same phenomenon. These forces will be called {\it Routh forces} and their corresponding systems will be named \textit{of Routh type} (Definition \ref{def:routh system}).
%We will study different symplectic structures on $Q\times Q$ associated to a forced discrete mechanical system; then, for a particular type of force, prove that a symplectic structure is preserved by the corresponding flow.
The second goal of
the paper is to illustrate the reduction process considered
in~\cite{C-F-T-Z}
as applied to an unforced symmetric discrete mechanical system, using an affine discrete connection constructed out of the discrete momentum map and reducing with respect to the isotropy subgroup corresponding to the (constant) value the discrete momentum map takes on the trajectories one is interested in. The result of
this procedure ---the discrete Routh reduction--- is a dynamical
system whose trajectories can be
identified with those of a forced discrete mechanical system;
furthermore, the resulting forces are of the same type considered
before, so that we conclude that there is a symplectic structure that
is preserved by the flow of that system. Unraveling the construction
we give a precise description of that symplectic structure, which
turns out to be the pullback of a certain
closed $2$-form on $T^*Q$, recovering the
results of~\cite{J-L-M-W} (which are known to mirror their continuous counterpart) in a not necessarily abelian setting.

The plan for the paper is as follows. In Section 2 we review some
standard definitions and results, including those of a forced discrete
mechanical system $(Q,L_d,f_d)$. We also introduce closed $2$-forms
$\omega^+$ and $\omega^-$ on $Q\times Q$ that, for regular systems,
are symplectic; in fact, that symplecticity is equivalent to the
regularity of the system. In particular, we prove that if the system is of Routh type, then $\omega^+$ and $\omega^-$
are preserved by the flow of the system (Proposition \ref{prop:regularity}). If, in addition, the system
is regular, those symplectic structures are seen to be the pullback
via the $\pm$-forced discrete Legendre transforms of a certain closed $2$-form on $T^*Q$.

In Section 3 we review the notion of symmetry of a forced discrete
mechanical system and the Lagrangian reduction of such systems, as
presented
in~\cite{C-F-T-Z}. This
is a general construction that requires the specification of some
data: a principal connection and an affine discrete connection on the principal
bundle $\pi:Q\lra Q/G$. The reduced system is, so far, just a
dynamical system. Then, under some conditions on the group actions, we
construct the affine discrete connection for the reduction process whose horizontal manifold is $J_d^{-1}(\{\mu\})$, for a value
$\mu\in\frakg^*$, where $J_d:Q\times Q\lra \frakg^*$ is the discrete momentum map and $\frakg:=\Lie{G}$. In what follows, we restrict the symmetry group to
$G_\mu$, the isotropy subgroup of $\mu$ under the coadjoint
action. Using these constructions, the reduced dynamical system can be
naturally identified as a \emph{forced} discrete mechanical system $(Q/G_\mu,\breve{L}_\mu,\breve{f}_\mu)$. We prove (Theorem \ref{hor-sym-theor-mu}) that the trajectories of the original symmetric system with constant discrete momentum $\mu$ can be fully recovered using those of the forced discrete mechanical system
obtained by the combined process described above. Interestingly, the
discrete force of the resulting system $\breve{f}_\mu$ can be
completely characterized in terms of the continuous and affine
discrete connections used in the reduction process, allowing us to
prove that $\breve{f}_\mu$ is a Routh force. We end this section comparing this with the results of \cite{J-L-M-W} and showing that the forced discrete mechanical system $(Q/G_\mu,\breve{L}_\mu,\breve{f}_\mu)$ may be interpreted as the discrete Routh reduction for the case of a not necessarily abelian Lie group.

Since $\breve{f}_\mu$ is a Routh force, the discussion of Section 2 implies the existence of a closed $2$-form $\breve{\omega}^+$ on $(Q/G_\mu)\times (Q/G_\mu)$ preserved by the flow of the reduced system, whenever it exists. In Section 4 we study this form and see that it coincides with the symplectic form $\omega_\mu^d$ obtained as the symplectic (or Marsden--Weinstein) reduction of $\omega_{L_d}$. This has two consequences: first, $\breve{\omega}^+$ is a
symplectic structure, so that the reduced system is regular and,
second, $\omega^d_\mu = \breve{\omega}^+$ is the pullback by the forced discrete
Legendre transform of the canonical symplectic form of $T^*(Q/G_\mu)$ modified by a magnetic term, a discrete analogue of the property that holds for (continuous) Routh reduction (Theorem \ref{thm:symplecticity and discrete routh reduction}).

Last, in Section 5, we report on a numerical experiment with a central potential in the plane. We apply the reduction procedure described in \cite{ar:marsden_ratiu_scheurle:2000:reduction_theory_and_the_lagrange-routh_equations} to the continuous system and use the fourth order Runge--Kutta integrator to approximante the trajectories of the reduced system. On the other hand, we discretized the (unreduced) continuous Lagrangian using a Mid Point scheme and then reduced the resulting discrete system following the ideas of the previous sections. The variational integrator is then constructed using the solution of the corresponding discrete equations of motion. The results suggest possible ideas for future work, namely the error analysis of the Routh reduction and the study of the commutativity of the reduction and discretization procedures. They are not the purpose of this paper, but we believe that they are very worthwhile subjects of future study on their own.

\section{Forced discrete mechanical systems}

We begin this section by recalling the basic definitions concerning forced discrete mechanical systems and their discrete evolution (for more detalis, see \cite{M-West}). Then we discuss conditions that guarantee the existence of flow for this systems and analyze the evolution of some symplectic structures.

\subsection{Preliminaries and definitions}

Given an $n$-dimensional manifold, we consider the product manifold
$Q \times Q$ and $\pr_1 : Q \times Q \lra Q$ and
$\pr_2 : Q \times Q \lra Q$ the canonical projections on the first and
second factor, respectively.  Using the product structure of
$Q \times Q$, we have that
$$T(Q \times Q) \simeq \pr_{1}^{\ast} (TQ) \oplus \pr_{2}^{\ast}
(TQ),$$ where $\pr_{i}^{\ast} (TQ)$ denotes the pullback of the
tangent bundle $TQ \lra Q$ over $Q \times Q$ by $\pr_i$ for
$i=1,2$. If we define $j_1 : \pr_{1}^{\ast} (TQ) \lra T(Q \times Q)$
as $j_1(\delta q) := (\delta q,0)$ we have that $j_1$ is an
isomorphism of vector bundles between $\pr_{1}^{\ast} (TQ)$ and
$TQ^{-} := ker(T \pr_2) \subset T(Q \times Q)$. Similarly, defining
$j_2 : \pr_{2}^{*} (TQ) \lra T(Q \times Q)$ as
$j_2(\delta q) := (0,\delta q)$ identifies $\pr_{2}^{*} (TQ)$ with the
subbundle $TQ^{+} := ker(T \pr_1) \subset T(Q \times Q)$.

So, the decomposition $T(Q \times Q) = TQ^{-} \oplus TQ^{+}$ leads to
the decomposition
$$T^{\ast}(Q \times Q) = (TQ^{-})^{\circ} \oplus (TQ^{+})^{\circ}$$
and the natural identifications
$$(TQ^{+})^{\circ} \simeq (TQ^{-})^{*} \simeq \pr_1^{*}T^{*}Q \text{ and } (TQ^{-})^{\circ} \simeq (TQ^{+})^{*} \simeq \pr_2^{*}T^{*}Q.$$

For any smooth map $H : Q \times Q \lra X$, where $X$ is a smooth
manifold, we define $D_{1}H := TH \circ j_1$ and
$D_{2}H := TH \circ j_2$, where $TH:T(Q \times Q) \lra TX$ denotes the
tangent map of $H$, as usual. Thus,
$$TH(q_0,q_1)(\delta q_0,\delta q_1) = D_{1}H(q_0,q_1)(\delta q_0) + D_{2}H(q_0,q_1)(\delta q_1).$$

If $H : Q \times Q \lra \rr$, then $D_{1}H(q_0,q_1) \in T^{*}_{q_0} Q$
and $D_{2}H(q_0,q_1) \in T^{*}_{q_1}Q$.

Naturally, these ideas may be extended to a product of more than two
(possibly different) manifolds.

\begin{definition}
A {\it forced discrete mechanical system} (FDMS) consists of a triple $(Q,L_d,f_d)$ where $Q$ is an $n$-dimensional differential manifold, the {\it configuration space}, $L_d : Q \times Q \lra \rr$ is a smooth map, the {\it discrete Lagrangian} and $f_d$ is a $1$-form on $Q \times Q$, the {\it discrete force}. 
\end{definition}

Using the identifications stated above, we will usually decompose forces
$f_d \in \Gamma(Q \times Q, T^{*}(Q\times Q))$ as
$f_d^{-} \oplus f_d^{+}$ with $f_d^{-} \in TQ^-$ and
$f_d^{+} \in TQ^+$, where $\Gamma(M,E)$ denotes the space of smooth
sections of a fiber bundle $E \lra M$.

Thus, $f_d^- (q_0,q_1) \in T_{q_0}^*Q$,
$f_d^+ (q_0,q_1) \in T_{q_1}^*Q$ and
$$f_d(q_0,q_1)(\delta q_0,\delta q_1) = f_d^- (q_0,q_1)(\delta q_0) + f_d^+ (q_0,q_1)(\delta q_1).$$

\begin{definition}
A {\it discrete curve} in $Q$ is a map $q. : \{ 0,\ldots,N \} \lra Q$ and an {\it infinitesimal variation} over a discrete curve $q.$ consists of a map $\delta q. : \{ 0,\ldots,N \} \lra TQ$ such that $\delta q_k \in T_{q_k}Q, \ \forall\, k=0,...,N $. An infinitesimal variation is said to have {\it fixed endpoints} if $\delta q_0 = \delta q_N = 0$.
\end{definition}

\begin{definition}
The {\it discrete action functional} of the FDMS $(Q,L_d,f_d)$ is defined as 
$${\displaystyle \frakS_d(q.) := \sum_{k=0}^{N-1} L_d(q_k,q_{k+1}).}$$
\end{definition}

The dynamics of a FDMS is given by the appropriately
modified discrete Lagrange d’Alembert principle known as the discrete Lagrange--d'Alembert principle. 

\begin{definition}
A discrete curve $q.$ is a {\it trajectory} of the FDMS $(Q,L_d,f_d)$ if it satisfies
$$\delta \left( \sum_{k=0}^{N-1} L_d(q_{k},q_{k+1})  \right) + \sum_{k=0}^{N-1} f_d(q_k,q_{k+1})(\delta q_k,\delta q_{k+1}) = 0,$$
for all infinitesimal variations $\delta q.$ of $q.$ with fixed endpoints. 
\end{definition}

The following well known result (see \cite{M-West}), which follows from the standard calculus of
variations, characterizes the trajectories of the system in terms of solutions of a set of algebraic
equations.

\begin{theorem}\label{thm:equations}
Let $(Q,L_d,f_d)$ a FDMS. Then, a discrete curve $q. : \{ 0,\ldots,N \} \lra Q$ is a trajectoy of $(Q,L_d,f_d)$ if and only if it satisfies the following algebraic identities
\begin{equation}\label{forcedELe}
D_2 L_d(q_{k-1},q_k) + D_1 L_d(q_k,q_{k+1}) + f_d^+(q_{k-1},q_k) + f_d^-(q_k,q_{k+1}) = 0 \in T_{q_k}^{*}Q
\end{equation}
for all $k=1,...,N-1$, called the {\it forced discrete Euler-Lagrange equations}.
\end{theorem}

As the case without forces, we can consider two fiber derivatives from a FDMS $(Q,L_d,f_d)$.

\begin{definition}
Let $(Q,L_d,f_d)$ be a FDMS. We define the maps $\ff^+_{f_d}L_{d}: Q \times Q \lra T^*Q$ and $\ff^-_{f_d}L_{d} : Q \times Q \lra T^*Q$ by
\[
\begin{split}
\ff^+_{f_d}L_{d}(q_0,q_1) & := (q_1, D_2 L_d(q_0,q_1) + f_d^+(q_0,q_1) ) \\
\ff^-_{f_d}L_{d}(q_0,q_1) & := (q_0, -D_1 L_d(q_0,q_1) - f_d^-(q_0,q_1) ),
\end{split}
\]
called {\it forced discrete Legendre transforms}.
\end{definition}

Note that the forced discrete Euler-Lagrange equations (\ref{forcedELe}) can be written as
$$\ff^+_{f_d}L_{d}(q_{k-1},q_{k}) =\ff^-_{f_d}L_{d}(q_{k},q_{k+1}) \ \ \ \forall\, k=1,...,N-1.$$

\begin{remark}\label {preservation}
Both $\ff^+_{f_d}L_{d}$ and $\ff^-_{f_d}L_{d}$ preserve the base points of the fiber bundles $\pr_{i} : Q \times Q \lra Q$ and $T^*Q \lra Q$, with $i=1,2$, respectively, in the sense that 
$$\ff^-_{f_d}L_{d}(q_0,q_1) \in T^{*}_{q_0}Q \ \ \mbox{and} \ \ \ff^+_{f_d}L_{d}(q_0,q_1) \in T^{*}_{q_1}Q.$$
That is, $\ff^-_{f_d}L_{d} \in \Gamma(Q \times Q, \pr_{1}^{*} (T^{*}Q))$ and 
$\ff^+_{f_d}L_{d} \in \Gamma(Q \times Q, \pr_{2}^{*} (T^{*}Q)).$
\end{remark}

\begin{remark}\label{remark:absorbing the force}
Suppose $f_d = d\gamma_d$, for some smooth function $\gamma_ d: Q \times Q \lra \rr$. We could then define a modified Lagrangian $L_{d}^{\gamma_d} : Q \times Q \lra \rr$ as $L_{d}^{\gamma_d} := L_d + \gamma_d$. By writing the equations of motion, it is immediately observed that the FDMS $(Q,L_d,f_d)$ is equivalent to the discrete mechanical system $(Q,L_{d}^{\gamma_d})$.
\end{remark}

\subsection{Regularity and existence of discrete flow}

As in the case without forces, the evolution of $(Q,L_d,f_d)$ will be described by a function whose domain and image will be included in $Q \times Q$ called \textit{forced discrete flow}, which may fail to be well-defined for arbitrary choices of FDMS. 

\begin{definition}
A FDMS $(Q,L_d,f_d)$ is said to be {\it regular} if both $\ff^+_{f_d}L_{d}$ and $\ff^-_{f_d}L_{d}$ are local isomorphisms of fiber bundles or, equivalently, if they are local diffeomorphisms.

In some special cases, such as when $Q$ is a vector space, both forced discrete Legendre transforms may be global diffeomorphisms. In this case, we say that the FDMS $(Q,L_d,f_d)$  is {\it hyperregular}.

Since $\ff^+_{f_d}L_{d}$ and $\ff^-_{f_d}L_{d}$ are fiber preserving, being diffeomorphisms is equivalent to the following maps being diffeomorphisms:
\[
\begin{split}
\phi_q^+ : Q \lra T^*_qQ \quad &\text{given by} \quad \phi_q^+(q') := \ff^+_{f_d}L_{d}(q',q) = D_2 L_d(q',q) + f_d^+(q',q) \\
\phi_q^- : Q \lra T^{*}_{q}Q \quad &\text{given by} \quad \phi_q^-(q') := \ff^-_{f_d}L_{d}(q,q') = -D_1 L_d(q,q') - f_d^-(q,q'),
\end{split}
\]
where $q,q' \in Q$.
\end{definition}

\begin{remark}\label{regularity1}
Note that given $\delta q_0 \in T_{q_0}Q$,
\begin{equation*}
\begin{split}
T_{q_0} \phi^+_{q_1}  (\delta q_0) &= T\ff^+_{f_d}L_{d}
(q_0,q_1)(\delta q_{0},0) = D_{1}\ff^+_{f_d}L_{d}(q_0,q_1)(\delta q_0) \\
&= D_1D_2 L_d(q_0,q_1) + D_1 f_d^+(q_0,q_1).
\end{split}
\end{equation*}

If we identify $T_{\phi^+_{q_1}(q_0)}T_{q_1}^*Q \simeq T_{q_1}^*Q$ (the fiber was fixed all along), we can regard the tangent map as a function $T_{q_0}\phi^+_{q_1} : T_{q_0}Q \lra T_{q_1}^*Q$. With this in mind, we can define the bilinear map
\[
\begin{split}
	D_1D_2 L_d(q_0,q_1) + &D_1 f_d^+(q_0,q_1) : T_{q_0}Q \times T_{q_1}Q \lra \rr \\
	(\delta q_0,\delta q_1) &\mapsto (T_{q_0}\phi_{q_1}^+ (\delta q_0))(\delta q_1) \in \rr.
\end{split}
\]

Furthermore, if $\phi^+_{q_1}$ is a local diffeomorphism, then $T_{q_0}\phi^+_{q_1}$ is an isomorphism and this bilinear map is non degenerate. Indeed, $(T_{q_0}\phi^+_{q_1}(\delta q_0))(\delta q_1) = 0,  \forall \ \delta q_1 \in T_{q_1}Q$, implies that $\delta q_0 \in \ker T_{q_0}\phi^+_{q_1} = \{ 0 \}$.

Analogously, if $\phi^-_{q_0}$ is a local diffeomorphism, then the bilinear map
$$D_2\ff^-_{f_d}L_{d}(q_0,q_1)=-D_2D_1 L_d(q_0,q_1) - D_2f_d^-(q_0,q_1): T_{q_1}Q \lra T_{q_0}^{*}Q$$
is non degenerate.

Therefore, the regularity of the system $(Q,L_d,f_d)$ is equivalent to the bilinear maps
$$D_1D_2 L_d(q_0,q_1) + D_1f_d^+(q_0,q_1): T_{q_0}Q \times T_{q_1}Q \lra \rr$$
$$-D_2D_1 L_d(q_0,q_1) - D_2f_d^-(q_0,q_1): T_{q_1}Q \times T_{q_0}Q     \lra \rr$$
being non degenerate. This can be expressed in terms of their associated matrices, which should be, locally, non-singular. Consider coordinates $(q_0^1,\ldots,q_0^n,q_1^1,\ldots,q_1^n)$ on $Q \times Q$. If $f_d^+ = f_i^+ \ dq_1^i$, then these matrices are
$$\frac{\partial f_d^+}{\partial q_0} =
\left(
\begin{array}{ccc}
\frac{\partial f_1^+}{\partial q_0^1} & \cdots & \frac{\partial f_1^+}{\partial q_0^n} \\
\vdots & \ddots & \vdots \\
\frac{\partial f_n^+}{\partial q_0^1} & \cdots & \frac{\partial f_n^+}{\partial q_0^n}
\end{array}
\right),
\quad
\frac{\partial^2 L_d}{\partial q_0 \partial q_1} =
\left(
\begin{array}{ccc}
\frac{\partial^2 L_d}{\partial q_0^1 \partial q_1^1} & \cdots & \frac{\partial^2 L_d}{\partial q_0^n \partial q_1^1} \\
\vdots & \ddots & \vdots \\
\frac{\partial^2 L_d}{\partial q_0^1 \partial q_1^n} & \cdots & \frac{\partial^2 L_d}{\partial q_0^n \partial q_1^n}
\end{array}
\right)$$
$$\frac{\partial^2 L_d}{\partial q_0 \partial q_1} + \frac{\partial f_d^+}{\partial q_0} =
\left(
\begin{array}{ccc}
\frac{\partial^2 L_d}{\partial q_0^1 \partial q_1^1} + \frac{\partial f_1^+}{\partial q_0^1} & \cdots & \frac{\partial^2 L_d}{\partial q_0^n \partial q_1^1} + \frac{\partial f_1^+}{\partial q_0^n} \\
\vdots & \ddots & \vdots \\
\frac{\partial^2 L_d}{\partial q_0^1 \partial q_1^n} + \frac{\partial f_n^+}{\partial q_0^1} & \cdots & \frac{\partial^2 L_d}{\partial q_0^n \partial q_1^n} + \frac{\partial f_n^+}{\partial q_0^n}
\end{array}
\right),$$
where we have omitted the point $(q_0,q_1)$ everywhere. Finally, while this is the $(+)$-case, similar matrices are the ones associated to the other bilinear map.
\end{remark} 

Now, we study the existence of local flows for regular FDMS. It is clear that to analyze the good definition of the flow it suffices to consider paths of length 2.

\begin{theorem}\label{thm:flow}
Let $(Q,L_d,f_d)$ be a regular FDMS. Given a trajectory $(q_0,q_1,q_2)$ of $(Q,L_d,f_d)$, there are open subsets $U, V \subset Q \times Q$ and a diffeomorphism ${\bf F}_{L_d,f_d} : U \lra V$ such that

\begin{enumerate}
\item $(q_0,q_1) \in U, \ (q_1,q_2) \in V$ and ${\bf F}_{L_d,f_d}(q_0, q_1) = (q_1,q_2)$.
\item For any $(\widetilde{q_0},\widetilde{q_1}) \in U$ if $\widetilde{q_2} := \pr_2({\bf F}_{L_d,f_d}(q_0, q_1))$, then $(\widetilde{q_0}, \widetilde{q_1},  \widetilde{q_2})$ is a trajectory of $(Q,L_d,f_d)$.
\item Any trajectory $(q'_0, q'_1, q'_2 )$ of $(Q,L_d,f_d)$ such that $(q'_0, q'_1) \in U$ and $(q'_1, q'_2) \in V$ satisfies that
$(q'_1, q'_2) = {\bf F}_{L_d,f_d}(q'_0, q'_1)$.
\end{enumerate}
\end{theorem}
\begin{proof}
Given a trajectory $(q_0,q_1,q_2)$ of $(Q,L_d,f_d)$, since $\ff^-_{f_d}L_{d}$ is a local diffeomorphism, there are open subsets $V' \subset Q \times Q$ and $W' \subs T^*Q$ such that $(q_1,q_2) \in V'$ and $\ff^-_{f_d}L_{d} \mid_{V'}^{W'}$ is a diffeomorphism. Since $(q_0,q_1,q_2)$ is a trajectory of $(Q,L_d,f_d)$, it satisfies equations \eqref{forcedELe}: $\ff^-_{f_d}L_{d} (q_1,q_2) = \ff^+_{f_d}L_{d} (q_0,q_1)$
and we see that $\ff^+_{f_d}L_{d} (q_0, q_1) \in W'$. The continuity of $\ff^+_{f_d}L_{d}$ implies that $U' := (\ff^+_{f_d}L_{d} )^{-1}(W') \subset Q \times Q$ is an open subset. Thus, for any $(\widetilde{q_0},\widetilde{q_1}) \in U'$, $\ff^+_{f_d}L_{d} (\widetilde{q_0},\widetilde{q_1}) \in W'$ and
\[
(\widetilde{q_1},\widetilde{q_2}) :=(\ff^-_{f_d}L_{d}  
|_{V'}^{W'})^{-1}((\ff^+_{f_d}L_{d} ))(\widetilde{q_0},\widetilde{q_1}) \in V'
\]
is well defined and satisfies
$\ff^-_{f_d}L_{d} (\widetilde{q_1},\widetilde{q_2}) = \ff^+_{f_d}L_{d} (\widetilde{q_0},\widetilde{q_1})$. Therefore, $(\widetilde{q_0},\widetilde{q_1},\widetilde{q_2})$ is a trajectory of $(Q,L_d,f_d)$.

The same argument as above, but using that $\ff^+_{f_d}L_{d}$ is a local diffeomorphism shows
that there are open subsets $V'', U'' \subset Q \times Q$ and $W'' \subs T^*Q$ such that $(q_0, q_1) \in U''$, $(q_1, q_2) \in V''$, $\ff^+_{f_d}L_{d} |_{U''}^{W''}$ is a diffeomorphism and $(\ff^-_{f_d}L_{d}  )^{-1}(W'') = V''$.
If we now define $V := V ' \cap V'', U := U' \cap U''$ and $W := W' \cap W''$, we have open subsets of $Q \times Q$ and $T^*Q$ such that $(q_0, q_1) \in U$ and $(q_1, q_2) \in W$. Furthermore, it is easy to check that
both $\ff^-_{f_d}L_{d} \mid_{V}^{W}$ and $\ff^+_{f_d}L_{d} \mid_{U}^{W}$ are diffeomorphisms. Then, if we define
${\bf F}_{L_d,f_d}: U \lra V$ by
\[
{\bf F}_{L_d,f_d} := (\ff^-_{f_d}L_{d} \mid_{V}^{W} )^{-1} \circ (\ff^+_{f_d}L_{d} \mid_{U}^{W} )
\]
we have that $(\widetilde{q_0},\widetilde{q_1}, \pr_2({\bf F}_{L_d,f_d}(\widetilde{q_0},\widetilde{q_1})))$ is a trajectory of $(Q,L_d,f_d) $ for any $(\widetilde{q_0},\widetilde{q_1}) \in U$.

Also, if $(q'_{0}, q'_1, q'_2 )$ is a trajectory of $(Q,L_d,f_d) $ such that $(q'_0, q'_1 ) \in  U$ and $(q'_1, q'_2 ) \in V$, by \eqref{forcedELe}
and the invertibility of $\ff^-_{f_d}L_{d} \mid_{V}^ {W} $ , we conclude that $(q'_1, q'_2 ) = {\bf F}_{L_d,f_d}(q'_0, q'_1 )$. 
\end{proof}

The map ${\bf F}_{L_d,f_d}$ defined in Theorem \ref{thm:flow} will be called the {\it local flow} of the regular FDMS $(Q,L_d,f_d)$.

In Theorem 5.1 of \cite{ar:borda_fernandez_grillo:2013:discrete_second_order_constrained_lagrangian_systems_first_results} the authors prove some of the results of Theorem \ref{thm:flow} from a slightly different perspective and they also show that the set of all (length 2) trajectories of a regular FDMS, if not empty, is an embedded submanifold of $Q \times Q \times Q$.

\begin{remark}
Before continuing with the analysis of forced systems, let us consider the case $f_d \equiv 0$. It is well known that variational calculus gives rise to the so-called {\it discrete Euler-Lagrange equations} that describe the dynamics of the system. They can be written as
\begin{equation}
D_2 L_d(q_{k-1},q_k) + D_1 L_d(q_k,q_{k+1})  = 0 \ \ \ \forall\, k=1,...,N-1.
\end{equation}

In this case, a discrete mechanical system $(Q,L_d)$ is regular if the discrete Lagrangian $L_d$ is regular (as usual, see \cite{M-West}); that is, 
for all $q_0, q_1 \in Q$ the discrete Legendre transforms 
$\ff^+ L_{d}: Q \times Q \lra T^*Q$ and $\ff^- L_{d} : Q \times Q \lra T^*Q$ given by
$\ff^+L_{d}(q_0,q_1) :=  D_2 L_d(q_0,q_1)  \in T_{q_1}^*Q$ and $\ff^-L_{d}(q_0,q_1) :=  -D_1 L_d(q_0,q_1)      
\in T_{q_0}^*Q $
are local diffeomorphisms.
\end{remark}

\subsection{Structures on $Q\times Q$}

Returning to the variational approach already mentioned in Theorem \ref{thm:equations}, as in the case without forces (see Theorem 1.3.1 of \cite{M-West}), one can define two 1-forms on $Q\times Q$ by 
\[
\begin{split}
	\theta_{f_d}^{+} (q_0,q_1) & := D_2 L_d(q_0,q_1) + f_d^+(q_0,q_1) \in T_{q_1}^{*}Q \\
	\theta_{f_d}^{-}(q_0,q_1) & :=  -D_1 L_d(q_0,q_1) - f_d^-(q_0,q_1) \in T_{q_0}^{*}Q.
\end{split}
\]

If $(q^1,\ldots,q^n)$ is a local coordinate system on $Q$ and
\[
f_d^+(q_0,q_1) = (f_d^{+})_{i}(q_0,q_1)\, dq_1^i, \quad f_d^-(q_0,q_1) = (f_d^{-})_{i}(q_0,q_1)\, dq_0^i,
\]
then we can write, omitting the application point,
$$\theta_{f_d}^+ = \left( \frac{\partial L_d}{\partial q_1^i} + (f_d^{+})^{i} \right) \ dq_1^i \ \ \ \mbox{and} \ \ \ \theta_{f_d}^- = \left(- \frac{\partial L_d}{\partial q_0^i} - (f_d^{-})^{i} \right) \ dq_0^i.$$

Taking exterior differential, we can define two 2-forms on $Q  \times Q$ as 
$$\omega_{f_d}^{+} := -d\theta_{f_d}^{+} \ \ \ \mbox{and} \ \ \ \omega_{f_d}^{-} := -d\theta_{f_d}^{-}.$$

Thus, the expression of $\omega_{f_d}^+$ in local coordinates is
\begin{equation*}
\begin{split}
\omega_{f_d}^+ &= -d\theta_{f_d}^+ \\
&= - \left( \frac{\partial^2 L_d}{\partial q_0^j \partial q_1^i} + \frac{\partial (f_d^{+})^{i}}{\partial q_0^j} \right) \ dq_0^j \wedge dq_1^i - \left( \frac{\partial^2 L_d}{\partial q_1^j \partial q_1^i} + \frac{\partial (f_d^{+})^{i}}{\partial q_1^j} \right) \ dq_1^j \wedge dq_1^i \\
&= - \left( \frac{\partial^2 L_d}{\partial q_0^j \partial q_1^i} + \frac{\partial (f_d^{+})^{i}}{\partial q_0^j} \right) \ dq_0^j \wedge dq_1^i - \frac{\partial (f_d^{+})^{i}}{\partial q_1^j} \ dq_1^j \wedge dq_1^i.
\end{split}
\end{equation*}

Notice that $\theta_{f_d}^{+} - \theta_{f_d}^{-}=dL_d+ f_d$, implying that
\begin{equation}\label{omegaf+ - omegaf-}
\omega_{f_d}^{+} -\omega_{f_d}^{-} = -df_d.
\end{equation}

Therefore, if $f_d$ is closed, we can define a unique 2-form on $Q\times Q$ by $\omega_{f_d} := \omega_{f_d}^{+}=\omega_{f_d}^{-}$.

\begin{remark}
When $f_d \equiv 0$, as discussed in Section 1.3.2 of \cite{M-West}, it is possible to define two 1-forms on $Q \times Q$ as $\theta^{+}(q_0,q_1) := D_2 L_d(q_0,q_1)$ and $\theta^{-}(q_0,q_1) :=- D_1 L_d(q_0,q_1)$. If the system is regular, they give rise to the symplectic structure on $Q \times Q$ given by
\[
\omega_{L_d}(q_0,q_1) := -d\theta^+(q_0,q_1) = -d\theta^-(q_0,q_1),
\]
whose expression in local coordinates is
\begin{equation}\label{eq:omegaLd}
\omega _{L_{d}}\left( q_{0},q_{1}\right) = -\frac{\partial ^{2}L_{d}}{%
\partial q_{0}^{i}\partial q_{1}^{j}}dq_{0}^{i}\wedge dq_{1}^{j}\text{.}
\end{equation}

Note that this definition of $\omega_{L_d}$ differs by a sign from the one found in \cite{M-West}, making the discrete Legendre transforms into symplectic maps.
\end{remark}

Hence, it is natural to wonder if  $\omega_{f_d}^{-}$ and $\omega_{f_d}^{+}$ are symplectic structures when the FDMS $(Q,L_d,f_d)$ is regular. Let us first recall that
\begin{equation}\label{2-eq:omega_f+ en coordenadas}
\omega_{f_d}^+ = -d\theta_{f_d}^+ = -\left( \frac{\partial^2 L_d}{\partial q_0^j \partial q_1^i} + \frac{\partial (f_d^{+})_{i}}{\partial q_0^j} \right) \ dq_0^j \wedge dq_1^i - \frac{\partial (f_d^{+})_{i}}{\partial q_1^j} \ dq_1^j \wedge dq_1^i,
\end{equation}
so that $\omega_{f_d}^+ = \omega_{L_d} - df_d^+$.

Similarly, $\omega_{f_d}^- = -d\theta_{f_d}^- = \omega_{L_d} + df_d^-$.

\begin{lemma}\label{SYMP}
Let $(Q,L_d,f_d)$ be a regular FDMS. Then the $2$-forms $\omega_{f_d}^+$ and $\omega_{f_d}^-$ are non degenerate.
\end{lemma}
\begin{proof}
If $(Q,L_d,f_d)$ is regular, as we observed in Remark \ref{regularity1}, the matrix associated to $D_1 D_2 L_d + D_1 f_d^+$ is invertible.

Remembering equation \eqref{2-eq:omega_f+ en coordenadas}, if $(\delta q_0,\delta q_1),(\dot{q}_0,\dot{q}_1) \in T_{(q_0,q_1)}(Q \times Q)$, then
\[
\begin{split}
\omega_{f_d}^+(q_0,q_1)&((\delta q_0,\delta q_1),(\dot{q}_0,\dot{q}_1)) = \\
&- \left( \frac{\partial^2 L_d}{\partial q_0^j \partial q_1^i} + \frac{\partial (f_d^+)_i}{\partial q_0^j} \right) (\delta q_0^j \dot{q}_1^i - \dot{q}_0^j \delta q_1^i) - \frac{\partial (f_d^+)_i}{\partial q_1^j} (\delta q_1^j \dot{q}_1^i - \dot{q}_1^j \delta q_1^i).
\end{split}
\]

Let us compute the matrix of $\omega_{f_d}^+$ taking bases of the tangent spaces associated to the local coordinates $(q^1,\ldots,q^n)$. Then, for $k,l = 1,\ldots,n$,
\begin{eqnarray*}
\omega_{f_d}^+(q_0,q_1) \left( \left( \frac{\partial}{\partial q^k} , 0 \right), \left( \frac{\partial}{\partial q^l} , 0 \right) \right) & = & 0 \\
\omega_{f_d}^+(q_0,q_1) \left( \left( \frac{\partial}{\partial q^k} , 0 \right), \left( 0 , \frac{\partial}{\partial q^l} \right) \right) & = & - \frac{\partial^2 L_d}{\partial q_0^k \partial q_1^l} - \frac{\partial (f_d^+)_l}{\partial q_0^k} \\
\omega_{f_d}^+(q_0,q_1) \left( \left( 0 , \frac{\partial}{\partial q^k} \right), \left( \frac{\partial}{\partial q^l} , 0 \right) \right) & = & \frac{\partial^2 L_d}{\partial q_0^l \partial q_1^k} + \frac{\partial (f_d^+)_k}{\partial q_0^l} \\
\omega_{f_d}^+(q_0,q_1) \left( \left( 0 , \frac{\partial}{\partial q^k} \right), \left( 0 , \frac{\partial}{\partial q^l} \right) \right) & = & - \frac{\partial (f_d^+)_l}{\partial q_1^k} + \frac{\partial (f_d^+)_k}{\partial q_1^l}.
\end{eqnarray*}

Hence, the matrix of $\omega_{f_d}^+$ is given by
$$[\omega_{f_d}^+] =
\left(
\begin{array}{cc}
A & B \\
C & D
\end{array}
\right)
=
\left(
\begin{array}{cc}
0 & B \\
-B^t & D
\end{array}
\right),$$
where
$$B := - \left( \frac{\partial^2 L_d}{\partial q_0 \partial q_1} + \frac{\partial f_d^+}{\partial q_0} \right)^t, \quad D := - \left( \frac{\partial f_d^+}{\partial q_1} \right)^t + \frac{\partial f_d^+}{\partial q_1}.$$

Applying a result of matrix algebra (see Theorem 2.2 of \cite{Lu-Shiou}), if $B$ is non singular, then the invertibility of $[\omega_{f_d}^+]$ is equivalent to that of $C - D B^{-1} A$. In this case, $A = 0$ and $C = -B^t$. Thus, it is clear that $[\omega_{f_d}^+]$ is non singular, since $B^t$ is non singular by hypothesis (recall Remark \ref{regularity1}).

The same reasoning may be used to conclude that $\omega_{f_d}^-$ is also non degenerate.
\end{proof}

\begin{proposition}\label{prop:unique omegaf}
Given a regular FDMS $(Q,L_d,f_d)$, the $2$-forms
\[
\omega_{f_d}^{+} := - d\theta_{f_d}^+ = \omega _{L_{d}} - d{f_d}^+ \quad \text{and} \quad \omega_{f_d}^{-} := -d\theta_{f_d}^- = \omega _{L_{d}} + d{{f_d}^-}
\]
are symplectic structures on $Q \times Q$.

In addition, if the force $f_d$ is closed\footnote{Notice that $f_d$ closed does not imply that $f_d^+$ and $f_d^-$ are closed.}, then $\omega_{f_d} := \omega_{f_d}^+ = \omega_{f_d}^- $ is a symplectic structure.
\end{proposition}
\begin{proof}
It is clear that $\omega_{f_d}^+$ and $\omega_{f_d}^-$ are closed. Since they are also non degenerate by the Lemma \ref{SYMP}, they are symplectic structures.

Finally, if $f_d$ is closed, then $\omega_{f_d} := \omega_{f_d}^+ = \omega_{f_d}^-$ is a symplectic structure (by the previous paragraph).
\end{proof}

Note that if $f_d \equiv 0$ and the discrete Lagrangian $L_d$ is regular, it is well known that the symplectic structure $\omega _{L_{d}}$ verifies $\displaystyle{\omega _{L_{d}}=\left( \mathbb{F}^{\pm }L_{d}\right) ^{\ast }(\omega_Q })$, where $\omega_Q$ is the canonical symplectic structure on $T^*Q$ (\cite{M-West}). Then, it is natural to study the relation between the forms 
$\omega_{f_d}^+$ and $\omega_{f_d}^-$ and the pullback by the forced discrete Legendre transforms of $\omega_Q$.

If $\theta_{Q}$ is the canonical 1-form on $T^*Q$ and  $\pi_{Q}:T^*Q\to Q$ is the canonical projection of the cotangent bundle, then
\[
\begin{split}
(\ff_{f_d}^{+}L_d)^*&(\theta_{Q})(q_{0},q_{1}) (\delta q_{0},\delta q_{1}) \\
&= \theta_Q(D_2 L_d(q_{0},q_{1}) +f^+_d(q_{0},q_{1}))(T_{(q_0,q_1)} \ff_{f_d}^{+}L_d (\delta q_{0},\delta q_{1})) \\
&= (D_2 L_d(q_{0},q_{1}) + f^+_d(q_{0},q_{1})) \left( T_{\ff_{f_d}^{+}L_d(q_0,q_1)}\pi_{Q}(T_{(q_0,q_1)} \ff_{f_d}^{+}L_d (\delta q_{0},\delta q_{1})) \right) \\
&\stackrel{\star}{=} (D_2 L_d(q_{0},q_{1}) +f^+_d(q_{0},q_{1})) T_{(q_0,q_1)} \pr_{2} (\delta q_{0},\delta q_{1}) \\
&= (D_2 L_d(q_{0},q_{1}) + f^+_d(q_{0},q_{1}))  (\delta q_{1}) \\
&= \theta_{f_d}^{+} (q_0,q_1)(\delta q_{1}).
\end{split}
\]
where, in $\star$, we have used that $\pi_{Q} \circ \ff_{f_d}^{+} L_d = \pr_{2}$.

Then, without writing the evaluation points, we have 
$(\ff_{f_d}^{+}L_d)^*(\theta_{Q} )= \theta_{f_d}^{+}$ and taking exterior differential we obtain that
\[
(\ff_{f_d}^{+}L_d)^*(-d\theta_{Q}) = -d\theta_{f_d}^{+}.
\]

Analogously, using that $ \pi_{Q} \circ \ff_{f_d}^{-}L_d = \pr_{1} $ one can verify that
\[
(\ff_{f_d}^{-}L_d)^*(-d\theta_{Q}) = -d\theta_{f_d}^{-}.
\] 

The previous computations prove the following result.

\begin{proposition}\label{prop:omegaf is pullback of omegaQ}
If $\omega_{Q}$ is the canonical symplectic structure on $T^*Q$, then
$$(\ff_{f_d}^{+}L_d)^*(\omega_{Q}) =\omega_{f_d}^{+} \ \ \ \ \mbox{and} \ \ \ \ (\ff_{f_d}^{-}L_d)^*(\omega_{Q}) =\omega_{f_d}^{-}.$$
\end{proposition}

\subsection{Preservation of structures on $Q\times Q$}
\label{Sec:forced_strctures}

In the unforced case, it is well known that the symplectic structure $\omega_{L_d}$ associated to a regular discrete mechanical system $(Q,L_d)$ is preserved by its discrete flow $\bF_{L_d}$ (see Section 1.3.2 of \cite{M-West}). Given a FDMS $(Q,L_d,f_d)$, we ask ourselves under which conditions the symplectic structures $\omega_{f_d}^{+}$ and $\omega_{f_d}^{-}$ are preserved by the flow $\bF_{L_d,f_d}$. In order to do this, we consider paths $q.=(q_0,q_1,q_2)$ of length 2.

Following the ideas in \cite{M-West}, we define $\hat{\mathfrak{S}}_d : Q \times Q \lra \rr$ as
$$\hat{\mathfrak{S}}_d(q_0,q_1) := S_d(q_0,q_1,q_2)=L_d(q_0,q_1)+L_d(q_1,q_2),$$
where $q.=(q_0,q_1,q_2)$ is the solution of (\ref{forcedELe}) with initial condition $(q_0,q_1)$.

Let us define a $1$-form $\frakS_{f_d}$ on $Q \times Q$ by
\[
\frakS_{f_d} := f_d + (\bF_{L_d,f_d})^* f_d.
\]

Explicitly,
\[
\frakS_{f_d}(q_0,q_1)(\delta q_0,\delta q_1) = f_d(q_0,q_1)(\delta q_0,\delta q_1) + f_d(q_1,q_2)(\delta q_1,\delta q_2),
\]
where $q. := (q_0,q_1,q_2)$ is the solution of \eqref{forcedELe} with initial condition $(q_0,q_1)$ and $\delta q_2 := T_{(q_1,q_2)}\pr_2 (T_{(q_0,q_1)}\bF_{L_d,f_{d}} (\delta q_0,\delta q_1))$.

Therefore, using the same notation, 
\[
\begin{split}
	d\hat{\frakS}_d&(q_0,q_1)(\delta q_0,\delta q_1) + \frakS_{f_d}(q_0,q_1)(\delta q_0,\delta q_1) \\
	& =  [D_1 L_d(q_1,q_{2}) + f_d^-(q_1,q_{2}) + D_2 L_d(q_{0},q_1) + f_d^+(q_{0},q_1)] \ \delta q_1 \\
	& \ \ \ \ \ + [D_1 L_d(q_0,q_1) + f_d^-(q_0,q_1)] \ \delta q_0 + [D_2 L_d(q_{1},q_2) + f_d^+(q_{1},q_2)] \ \delta q_2  \\
	& = [D_1 L_d(q_0,q_1) + f_d^-(q_0,q_1)] \ \delta q_0 + [D_2 L_d(q_{1},q_2) + f_d^+(q_{1},q_2)] \ \delta q_2  \\
	& =\theta_{f_d}^{+}(\mathbf{F}_{L_d,f_{d}}(q_0,q_1))(T_{(q_0,q_1)}\mathbf{F}_{L_d,f_{d}}  (\delta q_0,\delta q_1)) - \theta_{f_d}^{-}(q_0,q_1)(\delta q_0,\delta q_1)  \\
	& = (\mathbf{F}_{L_d,f_{d}})^*(\theta_{f_d}^{+})(q_0,q_1)(\delta q_0, \delta q_1) - \theta_{f_d}^{-}(q_0,q_1)(\delta q_0,\delta q_1).
\end{split}
\]

That is,
\[
d \hat{\frakS}_d + \frakS_{f_d} = (\bF_{L_d,f_d})^* \theta_{f_d}^+ - \theta_{f_d}^-.
\]

Taking exterior differential, we have
\begin{equation*}\label{dfrakSf}
-d \frakS_{f_d} = (\mathbf{F}_{L_d,f_{d}})^*(-d\theta_{f_d}^{+}) + d\theta_{f_d}^{-} = (\mathbf{F}_{L_d,f_{d}})^{*}(\omega_{f_d}^{+})-\omega_{f_d}^{-}
\end{equation*}
and, from the definition of $\frakS_{f_d}$ we have that
\[
d\frakS_{f_d} =  d(f_d + (\mathbf{F}_{L_d,f_{d}})^{*} f_d) = df_d + (\mathbf{F}_{L_d,f_{d}} )^{*} (df_d).
\]

Thus,
\[
- df_d - (\mathbf{F}_{L_d,f_{d}})^{*} (df_d) = (\mathbf{F}_{L_d,f_{d}})^{*}(\omega_{f_d}^{+})-\omega_{f_d}^{-} \stackrel{\eqref{omegaf+ - omegaf-}}{=} (\mathbf{F}_{L_d,f_{d}})^{*}(\omega_{f_d}^{+}) - \omega_{f_d}^{+} - df_d,
\]
which yields
\[
(\mathbf{F}_{L_d,f_{d}})^{*}(\omega_{f_d}^+)-\omega_{f_d}^+ = (\mathbf{F}_{L_d,f_{d}})^{*} (-df_d).
\]

Equivalently, by Proposition \ref{prop:omegaf is pullback of omegaQ},
\[
(\mathbf{F}_{L_d,f_{d}})^{*}((\mathbb{F}_{f_d}^{+} L_d)^*(\omega_Q)) - (\mathbb{F}_{f_d}^{+} L_{d})^*(\omega_Q) = (\mathbf{F}_{L_d,f_{d}})^{*} (-df_d).
\]

Note that this implies that the forced discrete flow $\mathbf{F}_{L_d,f_{d}}$ does not necessarily preserve the symplectic form $\omega^{+}_{f_d}$. However, we observe how this structure evolves with the flow, as states the following result, which is discrete counterpart of what happens in the continuous framework.

\begin{proposition}\label{cffN=2}
Let $(Q,L_d,f_d)$ be a regular FDMS. Then the evolution of the symplectic form $\omega_{f_d}^+$ by the forced discrete flow $\mathbf{F}_{L_d,f_{d}}$ is given by
\[
(\mathbf{F}_{L_d,f_{d}})^{*}(\omega_{f_d}^{+}) - \omega_{f_d}^{+} =  (\mathbf{F}_{L_d,f_{d}}) ^{*}  (-df_d).
\]

Explicitly, if $q. = (q_0,q_1,q_2)$ is the solution of \eqref{forcedELe} with initial condition $(q_0,q_1)$ and given $(\delta q_0,\delta q_1) \in T_{(q_0,q_1)}(Q \times Q)$, $\delta q_2 := T_{(q_1,q_2)}\pr_2 (T_{(q_0,q_1)}\bF_{L_d,f_{d}} (\delta q_0,\delta q_1))$, then
\[
\omega_{f_d}^+ (q_1,q_2)(\delta q_1,\delta q_2) = \omega_{f_d}^+ (q_0,q_1)(\delta q_0,\delta q_1) + df_d(q_1,q_2)(\delta q_1,\delta q_2).
\]
\end{proposition}

\begin{remark}
Analogously, the evolution of the symplectic form $\omega_{f_d}^-$ is given by
\[
(\mathbf{F}_{L_d,f_{d}})^{*}(\omega_{f_d}^{-}) - \omega_{f_d}^{-} =  (\mathbf{F}_{L_d,f_{d}}) ^{*}  (-df_d).
\]
\end{remark}

\begin{corollary}
If a FDMS $(Q,L_d,f_d)$ is regular and $f_d$ is closed, then the forced discrete flow $ \mathbf{F}_{L_d,f_{d}}$ preserves the symplectic structure $\omega_{f_d}$.
\end{corollary}
\begin{proof}
If $f_d$ is closed, then, by Proposition \ref{prop:unique omegaf}, there is a unique symplectic structure $\omega_{f_d}=\omega_{f_d}^+=\omega_{f_d}^-$. Finally, by the previous proposition,
\[
(\mathbf{F}_{L_d,f_{d}})^{*}(\omega_{f_d}) = \omega_{f_d}.
\]
\end{proof}

\begin{remark}
The statement of Proposition \ref{cffN=2} can be written equivalently, in terms of the canonical symplectic structure on $T^*Q$, as
\[
(\mathbf{F}_{L_d,f_{d}})^{*}((\mathbb{F}_{f_d}^{\pm}L_{d})^*(\omega_Q))-(\mathbb{F}_{f_d}^{\pm}L_{d})^*(\omega_Q)=(\mathbf{F}_{L_d,f_{d}})  ^{*}  (-df_d).
\]
\end{remark}

However, since in many cases the forces are not closed, it is interesting to analyze whether there are more general forces that, although not closed, give rise to symplectic structures preserved by the forced discrete flow. As we already said, in Section \ref{sec:routh} we consider the reduction of certain symmetries of a discrete mechanical system in a procedure that we call discrete Routh reduction. This reduction process gives rise to a FDMS whose forces are in general not closed, but possess certain characteristics that allow us to define (modifying the canonical structure of the cotangent bundle) a symplectic structure which results invariant by the forced discrete flow.

In what follows we will consider FDMS with forces that enjoy these characteristics and we will prove that they guarantee the definition of a symplectic structure preserved by the flow of the system.

Let us consider a regular FDMS $(Q,L_d,f_d)$ with forces $f_d$ satisfying that its exterior differential can be written as $-df_d = \pr_2^* \beta - \pr_1^* \beta$, where $\beta$ is a 2-form on $Q$ and $\pr_i : Q \times Q \lra Q$, $i=1,2$, is the projection on the $i$-th factor. In the first place, we are going to study the appearence of the forces that satisfy this condition. Then, we are going to show that this form $\beta$ can be used to modify the canonical structure on $T^*Q$ in such a way that the pullback by the forced discrete Legendre transforms of the modified structure gives rise to symplectic structures on $Q\times Q$ that are preserved by the forced discrete flow.

\begin{definition}\label{def:routh system}
Let $(Q,L_d,f_d)$ be a FDMS, we say that $f_d$ is a {\it Routh force} if there exists a 2-form $\beta \in \Omega^2(Q)$ such that $-df_d = \pr_2^* \beta - \pr_1^* \beta$. In this case, we say that the system is of {\it Routh type}, that $f_d$ is \emph{generated} by $\beta$ and that $\beta$ is its \emph{Routh potential}.
\end{definition}

\begin{proposition}\label{beta}
Let $f_d$ be a 1-form on $Q \times Q$ written as
\[  f_d(q_0,q_1)(\delta q_0,\delta q_1) = f_d^- (q_0,q_1)(\delta q_0) + f_d^+ (q_0,q_1)(\delta q_1) \]
where $f_d^- (q_0,q_1) \in T_{q_0}^*Q$, $f_d^+ (q_0,q_1) \in T_{q_1}^*Q$.
If $f_d$ is a Routh force generated by $\beta$, then in local coordinates $(q^1,\ldots,q^n)$,
\begin{enumerate}
\item $\displaystyle{\frac{\partial (f_d^+)^i}{\partial q_0^j}(q) = \frac{\partial (f_d^-)^j}{\partial q_1^i}(q)}$

\item $\displaystyle{\frac{\partial}{\partial q_1} \frac{\partial (f_d^-)^i}{\partial q_0^j} (q)= \frac{\partial}{\partial q_0} \frac{\partial (f_d^+)^i}{\partial q_1^j}(q) = 0}$

\item $\displaystyle{- \frac{\partial (f_d^-)^j}{\partial q_0^i}(q) = \frac{\partial (f_d^+)^j}{\partial q_1^i}(q)}$,
\end{enumerate}
where $f_d^+=(f_d^+)^i dq^i$ and $f_d^-=(f_d^-)^i dq^i$.

Moreover, $\beta$ is given by
\begin{equation}\label{eq:beta in coordinates}
\beta(q) = \frac{\partial (f_d^-)^j}{\partial q_0^i}(q) \ dq^j \wedge dq^i = -\frac{\partial (f_d^+)^j}{\partial q_1^i}(q) \ dq^j \wedge dq^i.
\end{equation}
\end{proposition}
\begin{proof}
Note that, since $df_d=df_d^{+}+ df_d^-$, then
\begin{equation*}
\begin{split}
df_d &= \frac{\partial (f_d^-)_i}{\partial q_0^j} \ dq_0^j \wedge dq_0^i + \frac{\partial (f_d^+)_i}{\partial q_1^j} \ dq_1^j \wedge dq_1^i + \frac{\partial (f_d^+)_i}{\partial q_0^j} dq_0^j \wedge dq_1^i \\
& \ \ \ \ \ + \frac{\partial (f_d^-)_i}{\partial q_1^j} \ dq_1^j \wedge dq_0^i \\
&= \frac{\partial (f_d^-)_i}{\partial q_0^j} \ dq_0^j \wedge dq_0^i + \frac{\partial (f_d^+)_i}{\partial q_1^j} \ dq_1^j \wedge dq_1^i \\
& \ \ \ \ \ + \left( \frac{\partial (f_d^+)_i}{\partial q_0^j} - \frac{\partial (f_d^-)_j}{\partial q_1^i} \right) \ dq_0^j \wedge dq_1^i.
\end{split}
\end{equation*}

Then, if $df_d$ can be written as $-df_d = \pr_2^* \beta - \pr_1^* \beta$, with $\beta$ a 2-form on $Q$,
\[
\pr_1^* \beta = \frac{\partial (f_d^-)_i}{\partial q_0^j} \ dq_0^j \wedge dq_0^i, \quad \pr_2^* \beta = -\frac{\partial (f_d^+)_i}{\partial q_1^j} \ dq_1^j \wedge dq_1^i,
\]
\[
\frac{\partial (f_d^+)_i}{\partial q_0^j} = \frac{\partial (f_d^-)_j}{\partial q_1^i}.
\]

On the one hand, since $\pr_2^* \beta$ and $\pr_1^* \beta$ are 2-forms on $Q\times Q$ we have that
\[
\frac{\partial}{\partial q_1} \frac{\partial (f_d^-)_i}{\partial q_0^j} = \frac{\partial}{\partial q_0} \frac{\partial (f_d^+)_i}{\partial q_1^j} = 0.
\]

On the other hand, $\displaystyle{\beta(q_0) = \pr_1^* \beta(q_0,q_1)}$ and $\displaystyle{\beta(q_1) = \pr_2^* \beta(q_0,q_1)}.$

Now, if $\displaystyle{\beta = \beta_{ji} \ dq^j \wedge dq^i}$,
\[
\beta_{ji}(q) = \frac{\partial (f_d^-)_i}{\partial q_0^j}(q) = - \frac{\partial (f_d^+)_i}{\partial q_1^j}(q)
\]
and, finally,
\[
\beta(q) = \frac{\partial (f_d^-)_i}{\partial q_0^j}(q) \ dq^j \wedge dq^i.
\]
\end{proof}

We will use this form $\beta$ to modify the form $\omega_Q$ in such a way that its pullback by the discrete Legendre transforms defines a symplectic structure in $Q\times Q$ preserved by the discrete flow $\mathbf{F}_{L_d,f_{d}}$.

We begin by noting that the evolution of $\pr^*_{1}\, \beta$ by $\mathbf{F}_{L_d,f_{d}}$ agrees with $\pr_2^{*}\, \beta$ on trajectories of the system $(Q,L_d,f_d)$.

\begin{lemma}\label{proyeccionesdebetaN=2}
If $q_\cdot=(q_0,q_1,q_2)$ is a solution of the forced discrete Euler-Largrange equations (\ref{forcedELe}), then
\[ \pr_2^{*}\, \beta =(\mathbf{F}_{L_d,f_{d}})^{*}( \pr^*_{1}\, \beta) \]
on this trajectory, that is,
\[
(\pr_2^* \beta)(q_0,q_1) = (\bF_{L_d,f_d})^* (\pr_1^* \beta)(q_0,q_1).
\]
\end{lemma}
\begin{proof}
Let $(q_0,q_1,q_2)$ be a solution of the forced discrete Euler--Largrange equations (\ref{forcedELe}). Then, given $(\delta q_0,\delta q_1) \in T_{(q_0,q_1)} (Q \times Q)$,
\[
(\mathbf{F}_{L_d,f_{d}})^{*} (\pr^*_{1} \beta)(q_0,q_{1})(\delta q_0,\delta q_{1}) = (\pr_1 \circ \ \mathbf{F}_{L_d,f_{d}} )^* \beta(q_0,q_{1}) (\delta q_0,\delta q_{1}) = \beta(q_1)(\delta q_1).
\]

On the other hand,
\[
\pr_2^{*} \beta (q_0,q_{1}) (\delta q_0,\delta q_{1}) = \beta (q_{1})(\delta q_{1}),
\]
and, comparing both equations, we have
\[
\pr_2^{*} \beta (q_0,q_{1}) = (\mathbf{F}_{L_d,f_{d}})^{*}(\pr^*_{1} \beta)(q_0,q_{1}).
\]
\end{proof}

\begin{lemma}
Let $(Q,L_d,f_d)$ be a FDMS of Routh type with potential $\beta \in \Omega^2(Q)$. Then, on the trajectories of the system,
\[  (\mathbf{F}_{L_d,f_{d}})^{*} \omega_{f_d}^{+} -  \omega_{f_d}^{+}=( \mathbf{F}_{L_d,f_{d}})^{*}(\pr_2^*\beta)-  \pr_2^*\beta. \]
\end{lemma}
\begin{proof}
As we have already observed in Proposition \ref{cffN=2},
\[ (\mathbf{F}_{L_d,f_{d}})^{*}\omega_{f_d}^{+} - \omega_{f_d}^{+}=(\mathbf{F}_{L_d,f_{d}})^*(-df_d).  \]

Now, since $f_d$ is a Routh force,
\[  (\mathbf{F}_{L_d,f_{d}})^{*}(\omega_{f_d}^{+}) - \omega_{f_d}^{+}  =   (\mathbf{F}_{L_d,f_{d}})^{*}(\pr^*_2\beta) - ( \mathbf{F}_{L_d,f_{d}})^{*}(\pr^*_1\beta).\]

Thus, by Lemma \ref{proyeccionesdebetaN=2}, we obtain that on trajectories of $(Q,L_d,f_d)$,
\[  (\mathbf{F}_{L_d,f_{d}})^{*}(\omega_{f_d}^{+}) - \omega_{f_d}^{+}  =   (\mathbf{F}_{L_d,f_{d}})^{*}(\pr^*_2\beta) -  \pr^*_2\beta. \]
\end{proof}

\begin{proposition}\label{prop:regularity}
Let $(Q,L_d,f_d)$ be a FDMS of Routh type with potential $\beta \in \Omega^2(Q)$. Then, the 2-form 
\[ \omega^+:= (\ff_{f_d}^+L_d)^* (\omega_Q - \pi_Q^* \beta)  \]
is preserved by the flow $\mathbf{F}_{L_d,f_{d}}$.
\end{proposition}
\begin{proof}
Since $\pi_Q \circ \ff_{f_d}^+L_d = \pr_2$,
\begin{equation}\label{eq:omega+ = omegaf+ - pr2beta}
\omega^+= (\ff_{f_d}^+L_d)^* (\omega_Q - \pi_Q^* \beta) =\omega^{+}_{f_d} - \pr_2^*\beta
\end{equation}
and the result is immediate by the previous lemma.
\end{proof}

\begin{remark}
A similar argument proves that
\[
\omega^- := (\ff_{f_d}^-L_d)^* (\omega_Q - \pi_Q^* \beta) = \omega_{f_d}^- - \pr_1^* \beta
\]
is a $2$-form preserved by the flow of the system.
\end{remark}

We have already seen that if a FDMS is of Routh type, there are two $2$-forms conserved by the flow of the system. We have yet to see the relation between their symplecticity and regularity of the system.

\begin{proposition}\label{regularity}
Let $(Q,L_d,f_d)$ be a FDMS of Routh type with potential $\beta \in \Omega^2(Q)$. The system is regular if and only if the $2$-forms
\[  \omega^+ := (\ff_{f_d}^+L_d)^* (\omega_Q - \pi_Q^* \beta), \quad \omega^- := (\ff_{f_d}^-L_d)^* (\omega_Q - \pi_Q^* \beta), \]
are symplectic structures on $Q \times Q$.
\end{proposition}
\begin{proof}
Notice that $\beta$ is closed. Indeed, recalling \eqref{eq:beta in coordinates},
\[
d\beta = \frac{\partial^2 (f_d^-)_i}{\partial q_0^k \partial q_0^j} \ dq_0^k \wedge dq_0^j \wedge dq_0^i = -\sum_i \underbrace{\left( \sum_{k,j} \frac{\partial^2 (f_d^-)_i}{\partial q_0^k \partial q_0^j} \ dq_0^k \wedge dq_0^j \right)}_{=0} \wedge dq_0^i = 0.
\]

Hence, both $\omega^+$ and $\omega^-$ are closed. Let us see that $\omega^+$ is non degenerate. We first recall that
\[
\omega^+ \stackrel{\eqref{eq:omega+ = omegaf+ - pr2beta}}{=} \omega_{f_d}^+ - \pr_2^* \beta.
\]

Since
\[
(\pr_2^* \beta)(q_0,q_1) = -\frac{\partial (f_d^+)_i}{\partial q_1^j}(q_1) \ dq_1^j \wedge dq_1^i,
\]
using the expression of $\omega_{f_d}^+$ (computed in \eqref{2-eq:omega_f+ en coordenadas}),
\[
\omega^+ = \omega_{f_d}^+ - \pr_2^* \beta = - \left( \frac{\partial^2 L_d}{\partial q_0^j \partial q_1^i} + \frac{\partial (f_d^+)_i}{\partial q_0^j} \right) \ dq_0^j \wedge dq_1^i.
\]

So, taking into account the notation used in Lemma \ref{SYMP}, the matrix of $\omega^+$ is given by
$$[\omega^+] =
\left(
\begin{array}{cc}
0 & B \\
-B^t & 0
\end{array}
\right),$$
and it is non singular under the same regularity conditions of Lemma \ref{SYMP}.

Reciprocally, the non degeneracy of the form $\omega^+$ implies that the matrix $[\omega^+]$ is invertible, which forces $B$ to be invertible, and this last condition is precisely the Legendre transform $\ff_{f_d}^+ L_d$ being a local diffeomorphism. The same arguments applied to $\omega^-$ complete the proof.
\end{proof}

We can now state the following result.

\begin{theorem}\label{theorem:omega+ is symplectic and preserved}
Let $(Q,L_d,f_d)$ be a regular FDMS of Routh type with potential $\beta \in \Omega^2(Q)$. Then, the flow $\mathbf{F}_{L_d,f_{d}}$ preserves the symplectic structures  
\[ \omega^+ = (\ff_{f_d}^+L_d)^* (\omega_Q - \pi_Q^* \beta), \  \  \omega^- = (\ff_{f_d}^-L_d)^* (\omega_Q - \pi_Q^* \beta).  \]
\end{theorem}

\begin{remark}
The term $\pi_{Q}^* \beta$ can be interpreted as a correction term for the canonical symplectic structure that allows its pullback by the forced discrete Legendre transform to be a conserved symplectic structure on $Q \times Q$. In this context, we may call it a \emph{magnetic term}, in analogy with the systems known as magnetic systems in the continuous setting (see, for example, Section 6.6 of \cite{M-R}).
\end{remark}

\section{Symmetries and reduction of Lagrangian discrete mechanical systems}

In this section we will present a reduction process for discrete mechanical systems with symmetries according to the reduction process established in \cite{C-F-T-Z} for forced discrete mechanical systems. Thus, in order to formulate the variational principle and write the equations of motion of the reduced system obtained by eliminating symmetries of a discrete mechanical system $(Q,L_d)$, we first present a summary of these results.

\subsection{Symmetries and reduced dynamics}
\label{sec:red_dyn}

Let $\frakA$ be a principal connection on the principal bundle $\pi : M \lra M/G$ with horizontal space $\Hor_\frakA$, so that $TM = \mathcal{V}^\SG \oplus \Hor_\frakA$, where the \jdef{vertical bundle} $\mathcal{V}^\SG$ over $M$ is a subbundle of $TM$ with fibers $\mathcal{V}^\SG_x:=T_x(l^M_\SG(\{ x \}))$. The {\it horizontal lift} associated to $\frakA$ is the map $h : M \times_{M/G} T(M/G) \lra TM$ defined as
$$ h(x,w_{\pi(x)}) =  h^x(w_{\pi(x)}) := v_x \ \text{if} \ v_x \in \Hor_\frakA(x) \ \text{and} \ T_x\pi ( v_x )= w_{\pi(x)}.$$

In Section 2.4 of \cite{C-M-R}, a principal connection $\frakA$ on a $G$-principal bundle $\pi : M \lra M/G$ is shown to induce an isomorphism of vector bundles over $M/G$, $\alpha_\frakA : TM/G \lra T(M/G) \oplus \widetilde{\frakg}$ given by
\[
\alpha_\frakA([v_x]) := T\pi (v_x) \oplus [x,\frakA(v_x)],
\]
where $\widetilde{\frakg} := (M \times \frakg)/G$ is the \emph{adjoint bundle}, with $G$ acting on $M$ by the action defining the $G$-principal bundle and on $\frakg$ by the adjoint action. $\widetilde{\frakg}$ is a vector bundle over $M/G$ with fiber $\frakg$ and projection induced by $\pr_1 : M \times \frakg \lra M$.

We now turn our attention to the discrete setting, where the product manifold $Q \times Q$ replaces the tangent bundle $TQ$. 

Let $G$ be a Lie group that acts on $Q$ freely and properly by a left action $l^Q$, making the quotient map $\pi:Q\lra Q/G$ a principal bundle with structure group $G$. We consider the diagonal action of $G$ on $Q\times Q$ defined by $l^{Q\times Q}_g(q_0,q_1):=(l^Q_g(q_0),l^Q_g(q_1))$.

A Lie group $G$ is a {\it symmetry group} of a FDMS $(Q,L_d,f_d)$ if $L_d$ is $G$-invariant by $l^{Q\times Q}_g$ and the force $f_d : Q \times Q \lra T^*Q$ is $G$-equivariant, considering on $T^*(Q \times Q)$ the cotangent lift of the diagonal action.

To establish a reduced variational principle and study the reduced dynamics, it is convenient to work on a diffeomorphic model of $(Q\times Q )/G$. As in the framework of continuos mechanical systems (see for example \cite{C-M-R}), in general, a reduced system is a dynamical system but not necessarily a mechanical one because it is not always possible to identify the space $(Q\times Q) /G$ with a product  manifold. However, it is worth pointing out that it is possbile to do so in some particular cases.

In \cite{F-T-Z}, the authors construct a model space for $(Q\times Q) /G$ associated to a geometric object on the principal bundle $\pi:Q\lra Q/G$ called affine discrete connection. Briefly, a discrete connection on $\pi:Q\lra Q/G$ consists in choosing a submanifold  of $Q \times Q$  with certain characteristics. Alternatively, this object can be described as a function on $Q\times Q$ with values in $G$ that satisfies certain properties\footnote{Actually, a discrete connection is defined on an open $\mathcal{U} \subset Q \times Q$ called domain of the discete connection, but in this paper we will consider that all discrete connections are defined on $Q\times Q$.}.

By using in addition a connection on the principal bundle, we are able to derive a reduced variational principle as well
as reduced equations of motion that split in horizontal and vertical parts.

Now, we make a brief review of affine discrete connections and some of its properties.

Recall that we can define the {\it discrete vertical bundle} for the $l^Q$ action of $G$ as the submanifold
$\calV^G_d := \{ (q,l_g^Q(q)) \in Q \times Q \mid q \in Q, \ g \in G \}$ and the {\it composition} of vertical and arbitrary elements of $Q \times Q$ is given by $\cdot : \calV^G_d \times_Q (Q \times Q) \lra Q \times Q$ with
\[
(q_0,l_g^Q(q_0)) \cdot (q_0,q_1) := (q_0,l_g^Q(q_1)),
\]
where $\calV^G_d \times_Q (Q \times Q)$ denotes bundle product of $\calV^G_d $ and $Q \times Q$ on $Q$.

In what follows we consider the action of $G$ on itself by conjugation and denote it by $l^G$, that is, $l_g^G(w) := g w g^{-1}$.

\begin{definition}
Let $\gamma : Q \lra G$ be a $G$-equivariant smooth function with respecto to $l^Q$ and $l^G$. An \emph{affine discrete connection} $\calA_d$ with \emph{level} $\gamma$ is a smooth function $\calA_d : Q \times Q \lra G$ satisfying
\begin{enumerate}
\item For all $q_0,q_1 \in Q$, $g_0,g_1 \in G$,
\[
\calA_d(l^Q_{g_0}(q_0) , l^Q_{g_1}(q_1)) = g_1 \calA_d(q_0,q_1) g_0^{-1}.
\]
		
\item $\calA_d(q,l^Q_{\gamma(q)}(q)) = e$, where $e$ is the neuter element of $G$.
\end{enumerate}
\end{definition}

An affine discrete connection $\calA_d$ on a $G$-principal bundle $\pi : Q \lra Q/G$ defines a \emph{horizontal submanifold} $\Hor_{\calA_d} \subs Q \times Q$ and a \emph{discrete horizontal lift} $h_d : Q \times Q/G \lra Q \times Q$ given by
\[
h_d^{q_0} (\tau_1) := (q_0,q_1) \ \text{si} \ (q_0,q_1) \in \Hor_{\calA_d} \ \text{y} \ \pi(q_1) = \tau_1.
\]

In this case, $\Hor_{\calA_d} = \calA_d^{-1}(\{ e \})$, where $e \in G$ is the neuter of the group. Finally, using the discrete horizontal lift, we define $\overline{h_d^{q_0}}(\tau_1) := q_1$, where $q_1$ is the element satisfying $h_d^{q_0}(\tau_1) = (q_0,q_1)$.

In complete analogy to the continuous case, an affine discrete connection is equivalently defined via the function $\calA_d$, the horizontal submanifold or the discrete horizontal lift.

Following the ideas found in Section 4.2. of \cite{F-T-Z}, just as a principal connection allows us to identify the quotient $(TM)/G$ with a different model, a discrete connection $\calA_d$ provides an isomorphism of fiber bundles over $Q/G$, $\Phi_{\calA_d} : (Q \times Q)/G \lra \widetilde{G} \times Q/G$, where $\widetilde{G} := (Q \times G)/G$ is the \emph{conjugate associated bundle} (in which $G$ acts on $Q$ and $G$ via $l^Q$ and $l^G$, respectively) and $\Phi_{\calA_d}$ is defined dropping to the quotient the $G$-equivariant map $\widetilde{\Phi}_{\calA_d} : Q \times Q  \lra Q \times G \times (Q/G)$ given by
\[
\widetilde{\Phi}_{\calA_d}(q_0,q_1) := (q_0,\calA_d(q_0,q_1),\pi(q_1)).
\]

Its inverse $\widetilde{\Psi}_{\calA_d} : Q \times G \times (Q/G) \lra Q \times Q$, is defined by
\[
\widetilde{\Psi}_{\calA_d}(q_0,w_0,\tau_1) = (q_0,\widetilde{F}_1(q_0,w_0,\tau_1)),
\]
where the function $\widetilde{F}_1 : Q \times G \times (Q/G) \lra G$ is given by
\[
\widetilde{F}_1(q_0,w_0,\tau_1) := l_{w_0}^Q \left( h_d^{q_0}(\tau_1) \right).
\]

If we call the quotient maps $\rho : Q \times G \lra \widetilde{G}$ and $\widetilde{\pi} : Q \times Q \lra (Q \times Q)/G$ and we define $\Upsilon : Q \times Q \lra \widetilde{G} \times Q/G$ as $\Phi_{\calA_d} \circ \widetilde{\pi}$, we have the following commutative diagram:
\begin{equation}\label{1-eq:red-spaces}
	\xymatrixcolsep{5pc}\xymatrixrowsep{4pc}\xymatrix{
		Q \times Q \ar[d]_{\widetilde{\pi}}
		\ar[r]^{\widetilde{\Phi}_{\calA_d}} \ar[dr]_{\Upsilon} & Q \times G \times (Q/G) \ar[d]^{\rho \times 1_{Q/G}} \\
		(Q\times Q)/G \ar[r]_{\Phi_{\calA_d}} & \widetilde{G} \times Q/G
	}
\end{equation}

We summarize the isomorphisms we will use in a single statement.

\begin{theorem}[Lemma 2.4.2 of \cite{C-M-R} and Proposition 4.19 of \cite{F-T-Z}]
Let $\calA_d$ be an affine discrete connection on $\pi : Q \lra Q/G$ and let $\widetilde{\frakA}$ be a principal connection on $\widetilde{\pi} : Q \times Q \lra (Q \times Q)/G$. There exist isomorphisms, in the corresponding categories,
\[
T(Q \times Q)/G \simeq T \left( (Q \times Q)/G \right) \oplus \widetilde{\frakg}, \quad (Q \times Q)/G \simeq \widetilde{G} \times Q/G.
\]
\end{theorem}

\begin{remark}\label{1-remark:Upsilon-bundle}
Since $\widetilde{\pi}$ is a $G$-principal bundle and $\Phi_{\calA_d}$ is a diffeomorphism, it is clear that $\Upsilon : Q \times Q \lra \widetilde{G} \times Q/G$ defines a $G$-principal bundle.
\end{remark}

If $G$ is a symmetry group of a FDMS $(Q,L_d,f_d)$, using the fact that $\widetilde{\Phi}_{\calA_d}$ is a diffeomorphism, we can define $\check{L}_d : Q \times G \times (Q/G) \lra \rr$ by $\check{L}_d := L_d \circ (\widetilde{\Phi}_{\calA_d})^{-1}$. Since $L_d$ is $G$-invariant, so is $\check{L}_d$, inducing a reduced Lagrangian $\hat{L}_d$ on $\widetilde{G} \times (Q/G)$ given by
$$\hat{L}_d(\rho(q_0,w_0),\tau_1) := \check{L}_d(q_0,w_0,\tau_1).$$

The following diagram shows all the relevant maps introduced so far.
\[
\xymatrixcolsep{5pc}\xymatrixrowsep{4pc}\xymatrix{
{} & {} & {\R} \\ {Q\times Q} \ar[urr]^{L_d} \ar[d]_{\ti{\pi}}
\ar[r]_{\ti{\Phi}_{\CD}} \ar[dr]_{\Upsilon} & {Q\times \SG\times (Q/\SG)}
\ar[ur]^(.3){\check{L}_d} \ar[d]^{\rho\times id} & {}\\ {(Q\times Q)/\SG}
\ar[r]_{\Phi_{\CD}} & {\RS} \ar@/_/[uur]_{\hat{L}_d} & {}}
\]

Finally, we can establish a reduction process for the $G$-symmetry of the system $(Q,L_d)$ by considering the reduction process for FDMS proved in \cite{C-F-T-Z} for the particular case $f_d \equiv 0$. The following theorem is a discrete version of the one found in \cite{C-M-R}.

\begin{theorem} \label{th:4_points-general}
Let $\calA_d$ and $\frakA$ be an affine discrete connection and a principal connection, respectively, on the $G$-principal bundle $\pi : Q \lra Q/G$. Let $q_\cdot$ be a discrete curve in $Q$ and let
\begin{equation*}
\begin{split}
\tau_k &:= \pi(q_k), \quad 1 \le k \le N, \\
w_k &:= \calA_d(q_k,q_{k+1}), \quad v_k := \rho(q_k,w_k), \quad 0 \le k \le N-1
\end{split}
\end{equation*}
be the corresponding discrete curves in $Q/\SG$, $\SG$ and $\ti{\SG}$. Then, given a discrete mechanical system $(Q,L_d)$ with symmetry group $\SG$, the following statements are equivalent.
\begin{enumerate}
\item \label{it:var_pple-general} $q_\cdot$ satisfies the variational principle
$dS_d(q_\cdot)(\delta q_\cdot) = 0$ for all vanishing end points infinitesimal variations $\delta q_\cdot$ over $q_\cdot$.

  \item \label{it:eq_lda-general} $q_\cdot$ satisfies the
    discrete Euler-Lagrange equations
    \begin{equation}
    	\label{eq:DEL}
    	D_2 L_d(q_{k-1},q_k) + D_1 L_d(q_k,q_{k+1}) = 0, \quad 1 \le k \le N-1.
    \end{equation}

  \item \label{it:red_var_pple-general}  $(\tau_\cdot,v_\cdot) $ satisfies 
  \[ \delta \left( \sum_{k=0}^{N-1} \hat{L}_d(v_k,\tau_{k+1}) \right) = 0 \]
  for all infinitesimal variations $(\delta v_\cdot, \delta \tau_\cdot)$ over $(\tau_\cdot,v_\cdot)$ that satisfy
  \begin{itemize}
  	\item [a)] $\delta \tau_k \in T_{\tau_k}(Q/G)$ for $k=1,\dots,N$,
  	\item [b)] $\xi. = (\xi_0,\ldots,\xi_N) \in \frakg^{N+1}$,
  	\item [c)] for $k=0,\dots,N-1$,
  	\begin{equation} \label{eq:delta_vk-def}
  		\begin{split}
  			\delta v_k :=& \; T_{(q_k,w_k)}\rho\big(\HLc{q_k}(\delta \tau_k),
  			T_{(q_k,q_{k+1})}\CD(\HLc{q_k}(\delta \tau_k),\HLc{q_{k+1}}(\delta
  			\tau_{k+1}))\big)
  			\\
  			& + T_{(q_k,w_k)}\rho\big((\xi_k)_Q(q_k),
  			T_{(q_k,q_{k+1})}\CD((\xi_k)_Q(q_k),(\xi_{k+1})_Q(q_{k+1}))\big),
  		\end{split}
  	\end{equation}
   where $h$ is the horizontal lift associated to $\frakA$, $\delta\tau_0 \in T_{\pi(q_0)}(Q/G)$,
   \item [d)] and the fixed endpoints conditions: $\delta\tau_0=0$, $\delta\tau_N=0$.
  \end{itemize}

  \item \label{it:red_lda_eq-general}
    $(v_\cdot,\tau_\cdot)$ satisfies the following conditions for each
    fixed $(v_{k-1},\tau_k,v_k,\tau_{k+1})$, with $1 \leq k \leq N-1$.
    \begin{itemize}
    \item $\phi=0$ for $\phi \in T_{\tau_k}(Q/G)$ given by 
    \[ \begin{split}
    	\phi & := D_1 \check{L}_d(q_k,w_k,\tau_{k+1}) \circ \HLc{q_k}
    	+ D_3 \check{L}_d(q_{k-1},w_{k-1},\tau_{k}) \\ 
    	& + D_2\check{L}_d(q_k,w_k,\tau_{k+1})
    	D_1\calA_d(q_k,q_{k+1}) \circ \HLc{q_k} \\
    	& + D_2\check{L}_d(q_k,w_k,r_{k+1}) D_2\calA_d(q_k,q_{k+1})\circ\HLc{q_k}.
    \end{split} \]
    
    \item $\psi(\xi_k)=0$ for all $\xi_k \in \frakg$, where $\psi \in \frakg^*$ is given by
    \[  \psi(\xi_k):= \left( D_2\check{L}_d(q_{k-1},w_{k-1},\tau_{k}) w_{k-1}^{-1}
    - D_2\check{L}_d(q_k,w_k,\tau_{k+1}) w_k^{-1} \right)(\xi_k) \]
    \end{itemize}
where we use the notation $\alpha w:=(R_{w^{-1}})^*(\alpha)$ for $w \in G$ and $\alpha \in T_q^* Q $, being $R_w$ the right traslation.
\end{enumerate}  
\end{theorem}

\begin{proof} 
See Theorem 3.11 in \cite{C-F-T-Z} with $f_d \equiv 0$.
\end{proof}

\begin{remark}\label{map Y}
Suppose that the characteristics of the system indicate that some of its trajectories (depending on initial data) live in a submanifold of $Q\times Q$  
that defines an affine discrete connection $\CD$ on the fiber bundle $\pi:Q \lra Q/G$. That is, some of the trajectories of the system belong to the horizontal space of some affine discrete connection $\CD$. In this case applying the reduction procedure using this connection yields a system on a manifold such that the reduced horizontal space is diffeomorphic to a product manifold $Q/G \times Q/G$. Thus, it is obtained that part of the reduced system is actually a discrete mechanical system.

In \cite{F-T-Z}, Chaplygin and horizontal symmetries were analyzed as a case of symmetries of nonholonomic discrete mechanical systems where this occurs. 
In Section \ref{sec:routh} we will consider the so called discrete Routh reduction as a particularly interesting case in which it also happens.
\end{remark}

\subsection{Discrete momentum map and connections}
\label{sec:SDMC}

We begin this section by recalling the definition of the discrete momentum map and some of its fundamental properties.

\begin{definition}
Given a symmetry group $G$ of $(Q,L_d)$, its  {\it discrete momentun map} $J_d:Q \times Q \lra  \frak g^*$ associated is defined as 
$$ J_d(q_0,q_1)\xi :=D_2 L_d(q_0,q_1)\xi_Q(q_1) = -D_1 L_d(q_0,q_1)\xi_Q(q_0)$$
where $\xi_Q$ is the infinitesimal generator of $\xi \in \frakg:=\text{Lie}(G)$.
\end{definition}

\begin{proposition}\label{momentun map}
Let $G$ be a group symmetry of $(Q,L_d)$ and $J_d$ its discrete momentun map associated. Then, 
\begin{enumerate}
\item $J_d$ is constant on trajectories $q.$ of $(Q,L_d)$. That is,
\[
J_d(q_{k-1},q_k)=J_d(q_k,q_{k+1})
\]
if $q.$ is a trajectory of the system.
\item $J_d$ is $G$-equivariant with respect to the diagonal action and the coadjoint action $Ad^{*}$ on  $\frak g^*$. That is, $ J_d(l_g^{Q\times Q}(q_0,q_1))= Ad^{*}_{g^{-1}} (J_d(q_0, q_1))$ for all $g \in G$ and all $(q_0,q_1) \in Q \times Q$.
\item  If $L_d$ is regular and $\mu \in \frak g^* $ is a regular value of $J_d$, the subset  $J_d^{-1}(\mu)\subset Q\times Q$ is a regular submanifold of $Q\times Q$.
\end{enumerate}
\end{proposition}
\begin{proof}
For the first two claims see \cite{M-West}, for the last one, see Lemmas 11.2, 11.4 and Proposition 11.8 in  \cite{F-T-Z}.
\end{proof}

Let $q.$ be a trajectory of $(Q,L_d)$ such that $J_d(q_0,q_1)=\mu$. Then, all the points of the trajectory must also lie on the submanifold $J_d^{-1}(\mu)$. Now we will analize if this submanifold can be the horizontal space of an affine discrete connection $\calA_d$ on the fiber bundle $\pi:Q\lra Q/G$. For this,  among other things, $J_{d}^{-1}(\mu)$ must be $G$-invariant.

By Proposition \ref{momentun map} we know that $J_d$ is $G$-equivariant respect to the diagonal action and the coadjoint action on  $\frak g^*$. Then, given $(q_0,q_1) \in J_{d}^{-1}(\mu)$, it is clear that
$l_g^{Q\times Q}(q_0,q_1) \in J_{d}^{-1}(\mu)$ if and only if $g \in G_\mu :=\{g\in  G:Ad^*_{g^{-1}}(\mu) = \mu\}\subs G$ is the isotropy subgroup of $\mu$.

Thus, $J_d^{-1}(\mu)$ is not $G$-invariant but it is $G_\mu$-invariant. 
Since $G_\mu$ is a symmetry group of  $(Q,L_d)$ that  acts freely and properly on $Q$, one can consider the quotient manifold $Q/G_\mu$, the fiber bundle  $\pi_\mu : Q \lra Q/G_\mu$ and the discrete momentun map associated to this symmetry $J_\mu : Q \times Q \lra \frakg_\mu^*$ given by $ J_\mu(q_0,q_1)\xi := -D_1 L_d(q_0,q_1)\xi_Q(q_0)$, for $ \xi \in \frakg_\mu$ where $\frakg_\mu$ is the Lie algebra of $G_\mu$. 

Consequently, we will see, under some additional assumption, that $J_\mu^{-1}(\mu)$ defines an affine discrete connection $\calA_\mu$ on the fiber bundle $\pi_\mu : Q \lra Q/G_\mu$. 

We recall that if $G$ is a symmetry group of $(Q,L_d)$, we say that $L_d$ is {\it $G$-regular} at $(q_0,q_1)\in Q\times Q$ if the restriction of the bilinear form $D_2 D_1 L_d (q_0,q_1): T_{q_0}Q\times T_{q_1}Q \lra \rr$ to $\mathcal{V}^G_{q_0}\times \mathcal{V}^G_{q_1}$ is nondegenerate (see \cite{J-L-M-W} and \cite{F-T-Z}).

\begin{definition}
Let $G_\mu$ be a symmetry group of $(Q,L_d)$ with $L_d$ regular and $G_\mu$-regular. We say that $G_\mu$ is a {\it group of	$\mu$-good symmetries} if, in addition, for each $q\in Q$ there is a unique $g\in G_\mu$ such that $ J_\mu(q,l^Q_g(q)) = \mu$. 
\end{definition}

\begin{proposition}\label{prop:mu_good_imp_aff_conn}
Let $G_\mu$ be a group of $\mu$-good symmetries of $(Q,L_d)$. Then, for all $(q_0,q_1)\in Q \times Q$ exists a unique $g\in G_\mu$ such that $J_\mu(q_0,l^Q_{g^{-1}}(q_1))=\mu$.  
\end{proposition}
\begin{proof}
See Proposition 11.11. \cite{F-T-Z}.
\end{proof}

Now, we can define the affine discrete connection associated to $ J_\mu$.

\begin{definition}
Given $\mu\in \frak g^*$, if  $G_\mu$ be a group of $\mu$-good symmetries of $(Q,L_d)$, let us define the map $\calA_\mu:Q\times Q\lra G_\mu$ as $\calA_\mu(q_0,q_1):=g$ where $g$ is the element of $G_\mu$ that appears in above proposition. Thus,
\[ \calA_\mu(q_0,q_1)=e   \Leftrightarrow  J_\mu(q_0,q_1)=\mu.  \]

This map defines an affine discrete connection (see Proposition 11.11 of \cite{F-T-Z}) and is called {\it affine discrete connection associated to the momentum map $ J_\mu$} and its horizontal space $Hor_{\calA_\mu}$ is $J_\mu^{-1}(\mu)$.
\end{definition}

\begin{example}\label{example}
Consider the discrete mechanical system $(Q,L_d)$ that arises as a discretization of the mechanical system given by a bar lying on a horizontal plane, where $Q := S^1 \times \rr^2 $ and the discrete Lagrangian is given by
$$L_d(q_0,q_1) := \frac{m}{2h} \left[ (x_1 - x_0)^2 + (y_1 - y_0)^2 \right] + \frac{J}{2h} (\varphi_1 - \varphi_0)^2,$$
where $m$ is the mass of the bar, $J$ is the moment of inertia and $h>0$ is a fixed time--step.
	
Since the iterated derivative $D_2 D_1 L_d(q_0,q_1)$ is given by
$$\frac{\partial^2 L_d}{\partial q_1 \partial q_0}(q_0,q_1) = -\diag \left( \frac{m}{h} , \frac{m}{h} , \frac{J}{h} \right),$$
and $L_d$ is regular, the symplectic form $\omega_{L_d}$ defined in \eqref{eq:omegaLd} is
\begin{eqnarray*}
&& \omega_{L_d}(q_0,q_1)((\delta q_0 , \delta q_1) , (\dot{q}_0 , \dot{q}_1)) = \\
&& \ \ \ \frac{m}{h} (\delta x_0 \dot{x}_1 - \dot{x}_0 \delta x_1) + \frac{m}{h} (\delta y_0 \dot{y}_1 - \dot{y}_0 \delta y_1) + \frac{J}{h} (\delta \varphi_0 \dot{\varphi}_1 - \dot{\varphi}_0 \delta \varphi_1),
\end{eqnarray*}
for $((\delta q_0 , \delta q_1) , (\dot{q}_0 , \dot{q}_1)) \in T_{(q_0,q_1)}(Q \times Q)$.
	
Consider the group $G = \text{SE}(2)$, which may be identified with $S^1 \times \rr^2$, with the operation defined as
$$(\alpha,a,b) \cdot (\beta,a',b') := (\alpha + \beta, a' \cos\alpha - b' \sin\alpha + a, a' \sin\alpha + b' \cos\alpha + b).$$
	
If $g = (\alpha,a,b)$ and $q = (\varphi,x,y)$, this group $G$ acts on $Q$ by left multiplication $l : G \times Q \lra Q$. Then,
$$l_g(q) := ( \varphi + \alpha, x \cos\alpha - y \sin\alpha + a , x \sin\alpha + y \cos\alpha + b ).$$

$\ast$ $L_d$ is $G$-invariant.

Notice that $G$ is a non abelian symmetry group and the discrete Lagrangian is $G$-invariant since
\[ L_d(q_0,q_1)  = \frac{m}{2h} \| (x_1-x_0,y_1-y_0) \|^2 + \frac{J}{2h} (\varphi_1 - \varphi_0)^2 \]
and $SE(2)$-invariance is equivalent to translation/rotation invariance. Also, the invariance of the angular term is direct.
	
A direct calculation shows that  $(l^{Q \times Q}_g)^{*} \omega_{L_d}=\omega_{L_d}$ for each $g \in G$.

$\ast$ The isotropy subgroup $G_\mu$.

Let us compute the isotropy subgroup $G_\mu$, for some $\mu \in \frakg^*$. Given $g = (\alpha,a,b) \in S^1 \times \rr^2$, we have
$$g^{-1} = (-\alpha , -(a \cos\alpha + b \sin\alpha) , a \sin\alpha - b \cos\alpha).$$
	
Recall that the adjoint representation $\Ad : G \lra \text{Aut}(\frakg)$ is given by $\Ad_g \xi := T_e I_g (\xi)$, where $I_g : G \lra G$ is the inner automorphism defined by $I_g(h) := g h g^{-1}$.
	
Let $\gamma(t) = (\alpha(t),a(t),b(t))$ be a curve in $G$ such that $\gamma(0) = (0,0,0)$ and $\dot{\gamma}(0) = \xi = (\xi_1,\xi_2,\xi_3)$. Then, a straightforward computation shows that
\[
\Ad_g \xi = \left( \xi_1 , b \xi_1 + \cos\alpha \xi_2 - \sin\alpha \xi_3 , - a \xi_1 + \sin\alpha \xi_2 + \cos\alpha \xi_3 \right).
\]
	
Therefore,
\[
\begin{split}
\Ad_{g^{-1}} \xi &= \left( \xi_1 , (a \sin\alpha - b \cos\alpha) \xi_1 + \cos\alpha \xi_2 + \sin\alpha \xi_3 , \right. \\
& \ \ \ \ \ \left. (a \cos\alpha + b \sin\alpha) \xi_1 - \sin\alpha \xi_2 + \cos\alpha \xi_3 \right).
\end{split}
\]
	
Taking $\mu = (\mu_1,\mu_2,\mu_3)$, this implies that the coadjoint action on $\frakg^*$ is given by
\[
\begin{split}
\langle \Ad_g^* \mu , \xi \rangle &= \left[ \mu_1 + \mu_2 (a \sin\alpha - b \cos\alpha) + \mu_3 (a \cos \alpha + b \sin\alpha) \right] \xi_1 \\
& + \left[ \mu_2 \cos\alpha - \mu_3 \sin\alpha \right] \xi_2 + \left[ \mu_2 \sin\alpha + \mu_3 \cos\alpha \right] \xi_3.
\end{split}
\]
%\begin{eqnarray*}
%\langle \Ad_g^* \mu , \xi \rangle & = & \mu_1 \xi_1 + \mu_2 \left[ (a \sin\alpha - b \cos\alpha) \xi_1 + \cos\alpha \xi_2 + \sin\alpha \xi_3 \right] \\
%& & + \mu_3 \left[ (a \cos\alpha + b \sin\alpha) \xi_1 - \sin\alpha \xi_2 + \cos\alpha \xi_3 \right] \\
%& = & \left[ \mu_1 + \mu_2 (a \sin\alpha - b \cos\alpha) + \mu_3 (a \cos \alpha + b \sin\alpha) \right] \xi_1 \\
%& & + \left[ \mu_2 \cos\alpha - \mu_3 \sin\alpha \right] \xi_2 + \left[ \mu_2 \sin\alpha + \mu_3 \cos\alpha \right] \xi_3.
%\end{eqnarray*}
	
This yields that $\Ad_g^* \mu = \mu$ if and only if

\[
\left\{
\begin{split}
\mu_1 &= \mu_1 + \mu_2 (a \sin\alpha - b \cos\alpha) + \mu_3 (a \cos\alpha + b \sin\alpha) \\
\mu_2 &= \mu_2 \cos\alpha - \mu_3 \sin\alpha \\
\mu_3 &= \mu_2 \sin\alpha + \mu_3 \cos\alpha
\end{split}
\right.
\]
	
These equations can be written in a more compact way as
\[
\left\{
\begin{split}
& \mu_1 = \mu_1 + \mu_2 (a \sin\alpha - b \cos\alpha) + \mu_3 (a \cos\alpha + b \sin\alpha) \\
& \left(
\begin{array}{c}
\mu_2 \\
\mu_3
\end{array}
\right) = \left(
\begin{array}{cc}
\cos \alpha & -\sin \alpha \\
\sin \alpha & \cos \alpha
\end{array}
\right) \left(
\begin{array}{c}
\mu_2 \\
\mu_3
\end{array}
\right).
\end{split}
\right.
\]
where it worth noticing that the matrix involved is simply a rotation matrix by an angle $\alpha$.

Let us consider some cases:
\begin{itemize}
\item [*] $(\mu_2,\mu_3)=(0,0)$. In this case, $\alpha$ is arbitrary and the first equation holds for every $(a,b)$ as well, yielding $G_\mu=G$.
\item [*]  $(\mu_2,\mu_3) \neq (0,0)$. In this case, the second equation implies that $\alpha=0 \in S^1$ and the first equation reads 
\[ \mu_1=\mu_1-\mu_2 b + \mu_3 a, \]
which says that $(a,b)$ must lie on the plane $-\mu_2 b + \mu_3 a =0$. Therefore, the isotropy subgroup is
\[ G_\mu = \{ (0,a,b) \in G \ | -\mu_2 b + \mu_3 a =0 \}.  \]
\end{itemize}

We will restrict our attention to the case where $\mu_3 = 0$ and $\mu_2 \neq 0$, which means that
$$G_\mu = \{ (0,a,0) \in G \mid a \in \rr \}.$$
	
$\ast$ $J_\mu$ and the affine discrete connection $\calA_\mu$.

Consider the action of $G_\mu$ on $Q$ given by the restriction of $l$ to $G_\mu$, i.e., $l^\mu : G_\mu \times Q \lra Q$ with $l^\mu_g(q) = l_g(q)$. Explicitly, if $g = (0,a,0) \in G_\mu$ and $q = (\varphi,x,y) \in Q$,
$$l_g^\mu(q) = (\varphi , x + a , y).$$

In addition, $\frakg_\mu \simeq \{ (0,\xi_2,0) \mid \xi_2 \in \rr \}$ and $\frakg_\mu^* \simeq \{ (0,\mu',0) \mid \mu' \in \rr \}$.
	
Now, noticing that $L_d$ is $G_\mu$-regular, let us check that $G_\mu$ is a group of $\mu$-good symmetries. Indeed, given $q \in Q$, the equation $J_\mu(q,l_g^\mu(q)) = \mu$ is
$$\frac{m}{h}(x + a - x) = \mu_2,$$
which has $a := \mu_2 h/m$ as its only solution.

Let us compute $J_\mu : Q \times Q \lra \frakg_\mu^*$. In order to do that, we need the infinitesimal generator $\xi_Q^\mu$ associated to this new action we are considering.

Given $q \in Q$ and $\xi = (0,\xi_2,0) \in \frakg_\mu$, $\xi_Q^\mu(q) = (0,\xi_2,0)$. Then, the momentum map is
$$J_\mu(q_0,q_1) \cdot \xi = \frac{m}{h} \left( x_1 - x_0 \right) \xi_2,$$
$$J_\mu(q_0,q_1) = \left( 0 , \frac{m}{h} (x_1 - x_0) , 0 \right) \in \frakg_\mu^*.$$
		
The discrete connection form $\calA_\mu$ is given by
\begin{equation}\label{eq:A_mu}
\calA_\mu(q_0,q_1) = \left( 0 , x_1 - x_0 - \frac{\mu_2 h}{m} , 0 \right),
\end{equation}
which can be checked computing
$$J_\mu \left( q_0 , l^\mu_{- \calA_d(q_0,q_1)}(q_1) \right) = \left( 0 , \frac{m}{h} \left( x_1 - x_1 + x_0 + \frac{\mu_2 h}{m} - x_0 \right) , 0 \right) = (0 , \mu_2 , 0) = \mu.$$

Therefore,
\[
\Hor_{\calA_\mu} = J_\mu^{-1}(\mu) = \left\{ (q_0,q_1) \in Q \times Q \mid \frac{m}{h} \left( x_1 - x_0 \right) = \mu_2 \right\}.
\]
\end{example}

Now we are going to consider the spaces and isomorphisms defined in Section \ref{sec:red_dyn} but associated to the action of the isotropy group $G_\mu$, instead of the group $G$, and the affine discrete connection $\calA_\mu$.

Also, we consider a connection $\frakA$ on the principal bundle $\pi_\mu : Q \lra Q/G_\mu$ with horizontal bundle $\Hor_\frakA$. Thus, $TQ = \mathcal{V}^{G_\mu} \oplus \Hor_\frakA$, where $\mathcal{V}^{G_\mu}$ is the vertical bundle and $h_\mu : Q \times_{Q/G_\mu} T(Q/G_\mu) \lra TQ$ is the horizontal lift.

By considering the action of $G_\mu$ on $Q \times G_\mu$ given by $l_g^{Q \times G_\mu}(q,w) := (l_g^Q(q),l_g^{G_\mu}(w))$, with $l_g^{G_\mu}(w) := g w g^{-1}$, one can define the quotient manifold $\widetilde{G}_\mu := (Q \times G_\mu)/G_\mu$ and the canonical quotient map denoted by $\rho_\mu : Q \times G_\mu \lra \widetilde{G}_\mu$. That is, $\widetilde{G}_\mu$ is the conjugate associated bundle (over $Q/G_\mu$), whose projection map will be denoted $\pr_{Q/G_\mu}$.

The diffeomorphisms associated to this connection $\calA_\mu$ are
\[
\begin{split}
\widetilde{\Phi}_{\mathcal{A}_{\mu}} : Q \times Q \lra Q \times G_{\mu} \times (Q/G_{\mu}), \quad &\widetilde{\Phi}_{\mathcal{A}_{\mu}}(q_0,q_1) := (q_0,\mathcal{A}_{\mu}(q_0,q_1), \pi_{\mu}(q_1)) \\
\widetilde{\Psi}_{\mathcal{A}_{\mu}} : Q \times G_{\mu} \times (Q/G_{\mu}) \lra Q \times Q, \quad &\widetilde{\Psi}_{\mathcal{A}_{\mu}}(q_0,w_0,\tau_1) := (q_0,\widetilde{F}_1^{\mu}(q_0,w_0,\tau_1)),
\end{split}
\]
where $\widetilde{F}_1^{\mu}:Q\times G_\mu\times (Q/G_\mu)\lra Q$ is defined as the map $\widetilde{F}_1$ by using the affine discrete connection $\calA_\mu$, together with
\[
\begin{split}
&\Phi_{\calA_{\mu}} : (Q \times Q)/G_\mu \lra \widetilde{G_\mu} \times Q/G_\mu, \\
&\Phi_{\calA_{\mu}} \left( [q_0,q_1]_\mu \right) := \left( \rho_\mu(q_0,\calA_{\mu}(q_0,q_1)) , \pi_\mu(q_1) \right).
\end{split}
\]

It is clear that if we restrict the diffeomorphims $\widetilde{{\Phi}}_{\calA_\mu}$ and $\Phi_{\calA_\mu}$ to the submanifolds $\Hor_{\calA_\mu}$ and $(\Hor_{\calA_\mu})/G_\mu$, respectively, we obtain the diffeomorphims
\[
\widetilde{\Phi}_{\mu} : \Hor_{\mathcal{A}_{\mu}} \lra Q \times \{ e \} \times Q/G_\mu, \quad \widetilde{\Phi}_{\mu}(q_0,q_1) := (q_0,e, \pi_{\mu}(q_1) )
\]
and
\[
\begin{split}
\Phi_\mu  : (\Hor_{\mathcal{A}_{\mu}})/G_{\mu} &\lra (Q \times \{ e \})/G_\mu \times Q/G_\mu, \\
\Phi_\mu  \left( [q_0,q_1]_\mu \right) &:= \left( \rho_\mu(q_0,e) , \pi_\mu(q_1) )\right).
\end{split}
\]

In order to see that the product manifold $(Q/G_\mu) \times (Q/G_\mu)$ is a model for the reduced space $(\Hor_{\calA_\mu})/G_\mu$ we prove the following result.

\begin{proposition}
The map $\calX_\mu : (Q \times \{ e \})/G_\mu \times Q/G_\mu \lra (Q/G_\mu) \times (Q/G_\mu)$ given by
$$\calX_\mu(\rho_\mu(q,e),\tau) := \left( \pr^{Q/G_\mu}(\rho_\mu(q,e)) , \tau \right) = (\pi_\mu(q),\tau),$$
where $\pr^{Q/G_\mu} : (Q \times \{ e \})/G_\mu \lra Q/G_\mu$ is the projection of the bundle $(Q \times \{ e \})/G_\mu$ over $Q/G_\mu$ with fiber $\{ e \}$  is a diffeomorphism, whose inverse is $\calY_\mu : (Q/G_\mu) \times (Q/G_\mu) \lra (Q \times \{ e \})/G_\mu \times Q/G_\mu$ given by
$\calY_\mu(\tau_0,\tau_1) := (\rho_\mu(q_0,e),\tau_1)$ for any $q_0 \in (\pi_\mu)^{-1}(\tau_0)$.
\end{proposition}
\begin{proof}
It is clear that $\calY_\mu$ and $\calX_\mu$ are well defined and are smooth, because $\rho_\mu$ and $\pr^{Q/G_\mu}$ are canonical projections. Let us check that they are mutually inverses.
	
If $(\rho_\mu(q,e),\tau) \in (Q \times \{ e \})/G_\mu \times Q/G_\mu$, then
$$(\calY_\mu \circ \calX_\mu)(\rho_\mu(q,e),\tau) = \calY_\mu(\pi_\mu(q),\tau) = \left( \rho_\mu(\widetilde{q},e) , \tau \right) = \left( \rho_\mu(q,e) , \tau \right),$$
for some $\widetilde{q} \in (\pi_\mu)^{-1}(\pi_\mu(q))$, and where we have used the fact that $\rho_\mu(\widetilde{q},e) = \rho_\mu(q,e)$, because $q$ and $\widetilde{q}$ lie on the same fiber and the action of $G_\mu$ on $\{ e \}$ is trivial.
	
On the other hand, if $(\tau_0,\tau_1) \in (Q/G_\mu) \times (Q/G_\mu)$, then, for some $q_0 \in (\pi_\mu)^{-1}(\tau_0)$,
\[ (\calX_\mu \circ \calY_\mu)(\tau_0,\tau_1) = \calX_\mu(\rho_\mu(q_0,e),\tau_1) = (\pi_\mu(q_0),\tau_1) = (\tau_0,\tau_1). \]
\end{proof}

\begin{corollary}\label{alpha-mu}
The map $\alpha_\mu := \calX_\mu \circ \Phi_\mu: J_\mu^{-1}(\mu)/G_\mu\lra (Q/G_\mu) \times (Q/G_\mu)$ given by 
$\alpha_\mu \left( [q_0,q_1]_\mu \right)=\left(\pi_\mu(q_0),\pi_\mu(q_1)    \right)$ is a diffeomorphism.
\end{corollary}
\begin{proof}
It is clear that $\alpha_\mu$ is a diffeormorphism because it is a composition of diffeomorphisms and
$$\alpha_\mu([q_0,q_1]_\mu)=\calX_\mu\left( \rho_\mu(q_0,e) , \pi_\mu(q_1) \right)=(\pi_\mu(q_0),\pi_\mu(q_1))$$
with inverse given by
\[
\delta_\mu := \Phi_\mu^{-1} \circ \calY_\mu : Q/G_\mu \times Q/G_\mu \lra J_\mu^{-1}(\mu)/G_\mu
\]
\[
\delta_\mu(\tau_0,\tau_1) = \widetilde{\pi}_\mu(h_{d,\mu}^{q_0}(\tau_1)),
\]
where $q_0 \in (\pi_\mu)^{-1}(\tau_0)$ and $h_{d,\mu}^{q_0}$ is the discrete horizontal lift associated to $\calA_\mu$.
\end{proof}

\begin{example}\label{example_1}
In the context of the Example \ref{example}, we have that the $G_\mu$-orbit is
\[
G_\mu \cdot q = \{ (\varphi, x + a ,y) \mid a \in \rr \}
\]
and the vertical space is
\[
\calV_q^{G_\mu} = \{ 0 \} \times \rr \times \{ 0 \}.
\]

We may take a principal connection whose horizontal space is
\[
\calH_q^\nu := \rr \times \Span\{ (\nu , 1) \},
\]
for a chosen $\nu \in \rr$. Then, for $\delta q = (\delta \varphi,\delta x,\delta y) \in T_qQ$,
\begin{equation}\label{eq:A_conec}
\frakA(\delta q) = (0,\delta x - \nu \delta y,0).
\end{equation}
		
Given $\delta \tau = (\delta \varphi,\delta y) \in T_\tau (Q/G_\mu)$, the horizontal lift is
\begin{equation}\label{eq:h_mu}
h_\mu^q(\delta \tau) := (\delta \varphi, \nu \delta y,\delta y).
\end{equation}
	
Recalling that the level set of $J_\mu$ is simply
$$J_\mu^{-1}(\mu) = \left\{ (q_0,q_1) \in Q \times Q \mid \frac{m}{h} \left( x_1 - x_0 \right) = \mu_2 \right\},$$
we have
\[
J_\mu^{-1}(\mu)/G_\mu \simeq S^1 \times \rr \times S^1 \times \rr \simeq (Q/G_\mu) \times (Q/G_\mu),
\]
with the projection given by $\widetilde{\pi}_\mu : J_\mu^{-1}(\mu) \lra J_\mu^{-1}(\mu)/G_\mu$,
\[
\widetilde{\pi}_\mu (q_0,q_1) = (\varphi_0,y_0,\varphi_1,y_1).
\]
\end{example}

\subsection{Discrete Routh reduction}\label{sec:routh}

Usually, Routh reduction is identified with a reduction process of a symmetry for a 
Lagrangian system given by an abelian Lie group. Also, it can be considered as the Lagrangian analogous of Marsden Weinstein's symplectic reduction process in the context of Hamiltonian mechanical systems.

In what follows we consider a discrete mechanical system $(Q,L_d)$ with a symmetry Lie group $G$ not necesarilly abelian such that $L_d$ is regular. In order to make the most of the momentum preservation we can apply the Theorem \ref{th:4_points-general} using the affine discrete connection $\calA_{\mu}$ and identifying the reduced spaces using Corollary \ref{alpha-mu}. The reduced objects are
\[
\begin{split}
\check{L}_\mu : Q \times \{ e \} \times Q/G_\mu \lra \rr, \quad &\check{L}_\mu(q_0,e,\tau_1) := L_d \circ (\widetilde{\Phi}_\mu)^{-1} \\
\hat{L}_{\mu} : (Q \times \{ e \})/G_\mu \times Q/G_{\mu} \lra \rr, \quad &\hat{L}_{\mu}(\rho_{\mu}(q_0,e),\tau_1) := \check{L}_{\mu}(q_0,e,\tau_1) \\
\breve{L}_{\mu} := \hat{L}_{\mu} \circ \calY_{\mu}: Q/G_{\mu} \times Q/G_{\mu} \lra \rr, \quad &\breve{L}_{\mu} (\tau_0,\tau_1):= \check{L}_{\mu}(q_0,e,\tau_1).
\end{split}
\]

In a diagram,
\[
\xymatrixcolsep{5pc}\xymatrixrowsep{4pc}\xymatrix{
& & \rr \\
J_\mu^{-1}(\mu) \ar[d]_-{\alpha_\mu \circ \widetilde{\pi}_\mu} \ar[r]^-{\widetilde{\Phi}_\mu} \ar@/^/[urr]^{L_d} & Q \times \{ e \} \times Q/G_\mu \ar[ur]^{\check{L}_\mu} \ar[d]_-{\rho_\mu \times 1_{Q/G_\mu}} & \\
Q/G_\mu \times Q/G_\mu \ar[r]^-{\calY_\mu} & (Q \times \{ e \})/G_\mu \times Q/G_\mu \ar@/_/[uur]_-{\hat{L}_{\mu}} &
}
\]
where in reality $L_d$ is $L_d|_{J_\mu^{-1}(\mu)}$ and $\widetilde{\pi}_\mu$ is $\widetilde{\pi}_\mu|_{J_\mu^{-1}(\mu)}$.

Then, taking into consideration the technical results of Lemmas 10.4 and 10.5 of \cite{F-T-Z}, we have the following version of Theorem \ref{th:4_points-general}.

\begin{theorem}\label{hor-sym-theor-mu}
Let $G_\mu$ be a group of $\mu$-good symmetries of the discrete mechanical system $(Q,L_d)$, for some $\mu \in \frak g^{*}$. Let $q.$ be a discrete curve in $Q$ with momentum $\mu$ and let $\tau_k := \pi_\mu(q_k)$ be the corresponding curve on $Q/G_\mu$. Then, the following statements are equivalent:
\begin{enumerate}
\item $q.$ satisfies the variational principle $dS_d(q.)(\delta q.) = 0$ for all vanishing endpoints variations $\delta q.$.
		
\item $q.$ satisfies the discrete Euler-Lagrange equations \eqref{eq:DEL}.
		
\item $\tau.$ satisfies the variational principle
\begin{equation}
d\breve{\frakS}_\mu(\tau.) = - \sum_{k=0}^{N-1} \breve{f}_\mu(\tau_k,\tau_{k+1}),
\end{equation}
for all vanishing endpoints variations $\delta \tau.$, where
$$\breve{\frakS}_\mu(\tau.) := \sum_{k=0}^{N-1} \breve{L}_\mu(\tau_k,\tau_{k+1})$$
and
\[
\breve{f}_\mu(\tau_0,\tau_1)(\delta \tau_0,\delta \tau_1) :=  D_2 \check{L}_\mu (q_0,e,\tau_1) \left( T_{(q_0,q_1)} \calA_{\mu}  (h_{\mu}^{q_0}(\delta \tau_0),h_{\mu}^{q_1}(\delta \tau_1)) \right).
\]
	
\item $\tau.$ satisfies the forced discrete forced Euler-Lagrange equations 
\begin{equation}\label{fdeEL}
D_1 \breve{L}_\mu (\tau_k,\tau_{k+1}) + D_2 \breve{L}_\mu (\tau_{k-1},\tau_k) + \breve{f}_\mu^+(\tau_{k-1},\tau_k) + \breve{f}_\mu^-(\tau_k,\tau_{k+1}) = 0,
\end{equation}
where
$$\breve{f}_\mu(\tau_k,\tau_{k+1}) (\delta \tau_k,\delta \tau_{k+1}) = \breve{f}_\mu^-(\tau_k,\tau_{k+1})(\delta \tau_k) + \breve{f}_\mu^+(\tau_k,\tau_{k+1})(\delta \tau_{k+1}).$$
\end{enumerate}
\end{theorem}

\begin{remark}
The previous result shows that the reduced system obtained under the reduction process is a FDMS $(Q/G_\mu,\breve{L}_\mu,\breve{f}_\mu)$. When this system has an additional symmetry a reduction process by stages can be applied, i.e. perform a new reduction process by the residual symmetry using, for example, Theorem 3.11 of \cite{C-F-T-Z}.
\end{remark}

\begin{remark}\label{remark:discrete vs cont routh reduction}
A similar scenario can be obtained in the continuous setting, as explored in \cite{ar:marsden_ratiu_scheurle:2000:reduction_theory_and_the_lagrange-routh_equations}. Considering the constructions of the Section III of that paper and taking the quotient always with respect to the isotropy subgroup $G_\mu$ (instead of considering the whole group $G$), part of the right--hand side of the variational expression of Theorem III.9 vanishes (for fixed--endpoints variations), implying that the equations of motion of the reduced system are those of a forced Lagrangian system, with the reduced Routhian playing the role of the Lagrangian and the force being obtained contracting with a $2$-form on the reduced space.
\end{remark}

There is a simpler way to compute the force $\breve{f}_\mu$, as shown by the following lemma.

\begin{lemma}\label{lemma:fmu en terminos de mu}
Given $(\delta \tau_0,\delta \tau_1) \in T_{(\tau_0,\tau_1)}((Q/G_\mu) \times (Q/G_\mu))$,
$$\breve{f}_\mu(\tau_0,\tau_1)(\delta \tau_0,\delta \tau_1)  = \left \langle \mu \ ,  T_{(q_0,q_1)} \calA_\mu  (h_\mu^{q_0}(\delta \tau_0),h_\mu^{q_1}(\delta \tau_1)) \right \rangle,$$
where $q_0 \in Q$ is such that $\pi_\mu(q_0) = \tau_0$ and $q_1 := \overline{h_{d,\mu}^{q_0}} (\tau_1)$.
\end{lemma}
\begin{proof}
By definition
\[
\breve{f}_\mu (\tau_0,\tau_1)(\delta \tau_0,\delta \tau_1)  =  D_2 \check{L}_\mu (q_0,e,\tau_1) \left( T_{(q_0,q_1)} \mathcal{A}_{\mu}  (h_{\mu}^{q_0}(\delta \tau_0),h_{\mu}^{q_1}(\delta \tau_1) \right),
\]
where $q_0 \in Q$ is such that $\pi_\mu(q_0) = \tau_0$ and $q_1 := \overline{h_{d,\mu}^{q_0}} (\tau_1)$.
	
We will see that for $(q_0,e,\tau_1) \in Q \times G_\mu \times (Q/G_\mu)$ and $\xi \in T_eG_\mu$,  $\check{L}_\mu$ satisfies that
$$D_2\check{L}_\mu(q_0,e,\tau_1) ( \xi) = \langle \mu , \xi \rangle$$
so that $\breve{f}_\mu (\tau_0,\tau_1)(\delta \tau_0,\delta \tau_1)  = \left\langle \mu  \ T_{(q_0,q_1)} \calA_\mu  (h_\mu^{q_0}(\delta \tau_0),h_\mu^{q_1}(\delta \tau_1)) \right\rangle.$

Let $\alpha(t) = (q_0,g(t),\tau_1)$ be a curve contained in $Q \times G_\mu \times (Q/G_\mu)$ such that $\alpha(0) = (q_0,e,\tau_1)$ y $\dot{\alpha}(0) = (0,\xi,0)$. Then,
\[
\begin{split}
D_2\check{L}_\mu(q_0,e,\tau_1) \cdot \xi &= \frac{d}{dt} \check{L}_\mu \circ \alpha(t) \bigg|_{t=0} \\
&= \frac{d}{dt} L_d \left( q_0 , l_{g(t)}^Q \left( \overline{h_{d,\mu}^{q_0}} (\tau_1) \right) \right) \bigg|_{t=0}
\end{split}
\]

Hence, calling $q_1 := \overline{h_{d,\mu}^{q_0}} (\tau_1)$,
\[
\begin{split}
D_2\check{L}_\mu(q_0,e,\tau_1) \cdot \xi &= D_2 L_d (q_0,q_1) \cdot \frac{d}{dt} l_{g(t)}^Q(q_1) \bigg|_{t=0} \\
&= D_2 L_d (q_0,q_1) \cdot \xi_Q(q_1) \\
&= J_\mu(q_0,q_1) \cdot \xi \\
&= \langle \mu , \xi \rangle,
\end{split}
\]
where the last equality is due to $(q_0,q_1) \in J_\mu^{-1}(\mu)$.
\end{proof}

\begin{example}\label{example_2}
Following with the calculations of the Example \ref{example_1} and using coordinates $\tau = (\varphi,y)$ for $Q/G_\mu$, we have that
\[
\begin{split}
\check{L}_d(q_0,e,\tau_1) &= \check{L}_d((\varphi_0,x_0,y_0),e,(\varphi_1,y_1)) = L_d \left( (\varphi_0,x_0,y_0) , \left( \varphi_1 , \frac{\mu_2 h}{m} + x_0 , y_1 \right) \right) \\
&= \frac{m}{2h} \left[ \left( \frac{\mu_2 h}{m} + x_0 - x_0 \right)^2 + (y_1 - y_0)^2 \right] + \frac{J}{2h} (\varphi_1 - \varphi_0)^2.
\end{split}
\]

It is clear that $\breve{L}_\mu : (Q/G_\mu) \times (Q/G_\mu) \lra \rr$ is given by
$$\breve{L}_\mu(\tau_0,\tau_1) = \frac{m}{2h} \left[ \left( \frac{\mu_2 h}{m} \right)^2 + (y_1 - y_0)^2 \right] + \frac{J}{2h} (\varphi_1 - \varphi_0)^2.$$

Therefore, if $\delta \tau_i \in T_{\tau_i}(Q/G_\mu)$ with $i = 0,1$, recalling \eqref{eq:A_mu} and \eqref{eq:h_mu}, we have that
$$T_{(q_0,q_1)} \calA_\mu  \left( h^{q_0}(\delta \tau_0) , h^{q_1}(\delta \tau_1) \right) = \left( 0 , \nu (\delta y_1 - \delta y_0) , 0 \right),$$
where $(\tau_0,\tau_1) = (\pi_\mu \times \pi_\mu)(q_0,q_1)$. Finally, by the previous lemma, this yields
$$\breve{f}_\mu(\tau_0,\tau_1)(\delta \tau_0 , \delta \tau_1) = \langle \mu , (0 , \nu(\delta y_1 - \delta y_0) , 0) \rangle = \mu_2 \nu (\delta y_1 - \delta y_0),$$
that is,
$$\breve{f}_\mu(\tau_0,\tau_1) = -\mu_2 \nu \ dy_0 + \mu_2 \nu \ dy_1.$$

Then,
$$\breve{f}_\mu^+(\tau_0,\tau_1) = (\mu_2 \nu,0,0), \quad \breve{f}_\mu^-(\tau_0,\tau_1) = (-\mu_2 \nu,0,0).$$	

According to Theorem \ref{hor-sym-theor-mu}, $\tau.$ is a trajectory if and only if it satisfies
$$D_1 \breve{L}_\mu (\tau_k,\tau_{k+1}) + D_2 \breve{L}_\mu (\tau_{k-1},\tau_k) + \breve{f}_\mu^+(\tau_{k-1},\tau_k) + \breve{f}_\mu^-(\tau_k,\tau_{k+1}) = 0.$$

This yields, for the $y$ coordinate,
$$\frac{m}{h} (y_k - y_{k-1}) - \frac{m}{h} (y_{k+1} - y_k) + \mu_2 \nu - \mu_2 \nu = 0,$$
which in turn implies that $y_{k+1} = 2 y_k - y_{k-1}$. Applying the same reasoning to the rest of the coordinates, we have that the reduced flow is given by
$${\bf F}_\mu(\tau_{k-1},\tau_k) = (\varphi_k,y_k,2 \varphi_k - \varphi_{k-1},2 y_k - y_{k-1}).$$
\end{example}

The statement of Theorem \ref{hor-sym-theor-mu} is strikingly similar to that of Theorem 2.3 of \cite{J-L-M-W}. To stablish a comparison, we first need a few definitions.

A principal connection $\frakA$ on $\pi_\mu : Q \lra Q/G_\mu$ induces a $1$-form $\frakA_\mu$ on $Q$ defined as
\begin{equation}\label{3-eq:frakA_mu}
	\frakA_\mu(q)(\delta q) := \langle \mu , \frakA(\delta q) \rangle.
\end{equation}

\begin{remark}
Notice that $\frakA_{\mu}$ is $G_\mu$-invariant. Indeed, given $q \in Q$, $\delta q \in T_qQ$ and $g \in G_\mu$,
\[
\frakA_{\mu}(l_g^Q(q))(Tl_g^Q(\delta q)) = \langle \mu , \frakA(Tl_g^Q(\delta q)) \rangle = \langle \mu , \Ad_g ( \frakA(\delta q)) \rangle \stackrel{\star}{=} \langle \mu , \frakA(\delta q) \rangle = \frakA_\mu(q)(\delta q),
\]
where, in $\star$, we have used that $g \in G_\mu$. However, $\frakA_\mu$ does not necessarily annihilate vertical vectors, and, therefore, it does not drop to a $1$-form on $Q/G_\mu$.
\end{remark}

If we call $\pr_i : Q/G_{\mu} \times Q/G_{\mu} \lra Q/G_{\mu}$ the projection on the $i$-th factor, $i = 1,2$, we have the $1$-form on $Q \times Q$ given by
\[
\scrA_\mu := \pr_2^* \frakA_\mu - \pr_1^* \frakA_\mu.
\]

This time, given a vertical vector $\xi_{Q \times Q}(q_0,q_1) = (\xi_Q(q_0),\xi_Q(q_1))$,
\[
\begin{split}
	\scrA_\mu (q_0,q_1)(\xi_Q(q_0),\xi_Q(q_1)) &= (\pr_2^* \frakA_\mu - \pr_1^* \frakA_\mu)(\xi_Q(q_0),\xi_Q(q_1)) \\
	&= \frakA_\mu(q_1)(\xi_Q(q_1)) - \frakA_\mu(q_0)(\xi_Q(q_0)) \\
	&= \langle \mu , \xi - \xi \rangle = 0.
\end{split}
\]

Then, following Section 2.4 of \cite{J-L-M-W}, the restriction of $\scrA_\mu$ to $J_\mu^{-1}(\mu)$ drops to a $1$-form on $J_\mu^{-1}(\mu)/G_\mu$ and this space can be identified with $Q/G_{\mu} \times Q/G_{\mu}$. That is, $\scrA_\mu$ drops to a $1$-form $\breve{\scrA}_\mu$ on $Q/G_\mu \times Q/G_\mu$.

With this notation, Theorem 2.3 of \cite{J-L-M-W} may be stated in the following way.

\begin{theorem}[Theorem 2.3 of \cite{J-L-M-W}]\label{4-teor:red-routh-Leok}
Let $G$ be an abelian Lie group and let $\tau.$ be a discrete curve on $Q/G$ and $q.$ a discrete curve on $Q$ with momentum $\mu$ obtained lifting $\tau.$. The following statements are equivalent:
\begin{enumerate}
\item $q.$ satisfies the discrete Hamilton principle.
		
\item $q.$ satisfies the discrete Euler--Lagrange equations \eqref{eq:DEL}.
		
\item $\tau.$ satisfies the variational principle
\[
d\widetilde{\frakS}_\mu(\tau.) = \sum_{k=0}^{N-1} \breve{\scrA}_\mu(\tau_k,\tau_{k+1})
\]
for all infinitesimal variation $\delta \tau.$ with vanishing endpoints, where
\[
\widetilde{\frakS}_\mu(\tau.) := \sum_{k=0}^{N-1} \widetilde{L}_\mu(\tau_k,\tau_{k+1})
\]
and $\widetilde{L}_\mu : Q/G \times Q/G \lra \rr$ is $L_d|_{J_\mu^{-1}(\mu)}$ dropped to the quotient.

\item $\tau.$ satisfies the equations
\[
D_2 \widetilde{L}_\mu (\tau_{k-1},\tau_k) + D_1 \widetilde{L}_\mu (\tau_k,\tau_{k+1}) = \breve{\scrA}_\mu^+(\tau_{k-1},\tau_k) + \breve{\scrA}_\mu^+(\tau_k,\tau_{k+1}),
\]
where
\[
\breve{\scrA}_\mu(\tau_k,\tau_{k+1}) (\delta \tau_k,\delta \tau_{k+1}) = \breve{\scrA}_\mu^-(\tau_k,\tau_{k+1})(\delta \tau_k) + \breve{\scrA}_\mu^+(\tau_k,\tau_{k+1})(\delta \tau_{k+1}).
\]
\end{enumerate}
\end{theorem}

\begin{remark}
In the previous theorem we have used $G$ instead of $G_\mu$ because in \cite{J-L-M-W} the authors reduce by an abelian symmetry group.
\end{remark}

We will dedicate the rest of the section to showing that the similarity of the theorems is not a coincidence, but in fact $\widetilde{L}_\mu = \breve{L}_\mu$ and $\breve{\calA}_\mu = - \breve{f}_\mu$, so that Theorem \ref{hor-sym-theor-mu} may be interpreted as the discrete Routh reduction for the case of a not necessarily abelian symmetry group.

\begin{lemma}\label{lemma:adjunta-conexion discreta es diferencia de la conexion principal}
Let $G_\mu$ be a group of $\mu$-good symmetries of $(Q,L_d)$ for some $\mu \in \frak g^{*}$. If $\calA_\mu$ is the affine discrete connection associated to $J_\mu$ and $\frakA$ is a principal connection on $\pi_\mu : Q \lra Q/G_\mu$, given $(\delta q_0,\delta q_1) \in T_{(q_0,q_1)} J_\mu^{-1}(\mu)$, there exists $g_{1} \in G_\mu$ such that
$$\Ad_{g_{1}^{-1}} \left( T_{(q_0,q_1)} \calA_\mu  (h_\mu^{q_0}(\delta \tau_0),h_\mu^{q_1}(\delta \tau_1)) \right) = \frakA(\delta q_0) - \frakA(\delta q_1)$$
where $Ad_g$ is the adjoint representation and $\delta \tau_i := T_{q_i}\pi_\mu ( \delta q_i )$ with $i=1,2$.
\end{lemma}
\begin{proof}
We work in a local trivialization of the $G_\mu$-principal bundle $\pi_\mu:Q\lra Q/G_\mu$. Hence, $Q = (Q/G_\mu) \times G_\mu$, $q = (\tau,g)$, $\delta q = (\delta \tau,\delta g)$ and the principal connection is given by
$$\frakA(\delta q) = \widetilde{\frakA}(\tau) \delta \tau + g^{-1} \cdot \delta g$$
where $g^{-1} \cdot \delta g = T_g L_{g^{-1}} ( \delta g)$ (where $L_g$ denotes the left translation of $G$ on itself) and $\widetilde{\frakA} : T_\tau(Q/G_\mu) \lra \frakg_\mu$ is a linear map.
	
In this context, the involved elements are
\begin{eqnarray*}
h_\mu^{q_i}(\delta \tau_i) & = & (\tau_i,g_i,\delta \tau_i,\delta g_i), \quad \frakA(h_\mu^{q_i} (\delta \tau_i)) = 0  \\
(q_i,\delta q_i) & = & (\tau_i,g_i,\delta \tau_i,h_i), 
\end{eqnarray*}
for some $h_i \in T_{g_i}G_\mu$, $i = 0,1$. 

Since $\widetilde{\frakA}(\tau_i) \delta \tau_i = - g_i^{-1} \cdot \delta g_i$, for $i = 0,1$,
\begin{equation}\label{eq:breveF-calA-1}
\begin{split}
\frakA(\delta q_1) - \frakA(\delta q_0) &= \widetilde{\frakA}(\tau_1) \delta \tau_1 + g_1^{-1} \cdot h_1 - \widetilde{\frakA}(\tau_0) \delta \tau_0 - g_0^{-1} \cdot h_0 \\
&= -g_1^{-1} \cdot \delta g_1 + g_1^{-1} \cdot h_1 - (- g_0^{-1} \cdot \delta g_0) - g_0^{-1} \cdot h_0 \\
&= g_1^{-1} \cdot (h_1 - \delta g_1) - g_0^{-1} \cdot (h_0 - \delta g_0).
\end{split}
\end{equation}
	
Consider curves $q_i^h(t) = (\tau_i^h(t),g_i^h(t))$ such that $q_i^h(0) = q_i$ and $\dot{q}_i^h(0) = h_\mu^{q_i}(\delta \tau_i)$. Then, since $J_\mu^{-1}(\mu) = \Hor_{\calA_\mu}$, by definition of affine discrete connection,
\[
\left( q_0^h(t) , l_{\calA_{\mu}(q_0^h(t),q_1^h(t))^{-1}}^Q (q_1^h(t)) \right) \in J_\mu^{-1}(\mu).
\]
and we have a curve contained in $J_\mu^{-1}(\mu)$ given by
\[
\gamma(t) := \left( (\tau_0^h(t) , g_0^h(t)) , (\tau_1^h(t),\mathcal{A}_{\mu}(q_0^h(t),q_1^h(t))^{-1} g_1^h(t)) \right).
\]
	
Note that $\gamma(0) = (q_0,q_1)$ and $\dot{\gamma}(0) \in T_{(q_0,q_1)}J_\mu^{-1}(\mu)$ projects to $(\delta \tau_0,\delta \tau_1)$. An expression for $\dot{\gamma}(0)$ is
\[
\dot{\gamma}(0) = \left( (\delta \tau_0,\delta g_0) , \left( \delta \tau_1, \frac{d}{dt} \mathcal{A}_{\mu}(q_0^h(t),q_1^h(t))^{-1} g_1^h(t) \bigg|_{t=0} \right) \right).
\]

In addition, following the notation presented at the beginning, $h_0 = \delta g_0$ and
\[
h_1 = \frac{d}{dt} \mathcal{A}_{\mu}(q_0^h(t),q_1^h(t))^{-1} g_1^h(t) \bigg|_{t=0}.
\]
	
If we denote by $m : G_\mu \times G_\mu \lra G_\mu$ the multiplication in $G_\mu$ and
\[
\dot{a} := T_{(q_0,q_1)} \calA_{\mu}  (h_{\mu}^{q_0}(\delta \tau_0),h_{\mu}^{q_1}(\delta \tau_1)),
\]
we have (see, for example, \cite{Michor}, Lemma 4.2)
\[
\begin{split}
\frac{d}{dt} \mathcal{A}_{\mu}(q_0^h(t),q_1^h(t))^{-1} g_1^h(t) \bigg|_{t=0} &= T_{(e,g_1)} m (- \dot{a},\delta g_1) \\
&= - T_e R_{g_1} (\dot{a}) + T_{g_1} L_e (\delta g_1) \\
&= - T_e R_{g_1} (\dot{a}) + \delta g_1,
\end{split}
\]
where $L$ and $R$ are the left and right translations on the Lie group $G_\mu$, respectively, and we have used that if $\nu : G_\mu \lra G_\mu$ is the inversion $g \mapsto g^{-1}$, then $T_e\nu = - 1_{T_eG_\mu}$.
	
Finally, combining this last equation with (\ref{eq:breveF-calA-1}),
\[
\begin{split}
\frakA(\delta q_1) - \frakA(\delta q_0) &= T_{g_1} L_{g_1^{-1}} \left( - T_e R_{g_1} (\dot{a}) + \delta g_1 - \delta g_1 \right)  - T_{g_0}L_{g_0^{-1}}  \left( \delta g_0 - \delta g_0 \right) \\
&= T_{g_1} L_{g_1^{-1}}  \left( - T_e R_{g_1} (\dot{a}) \right) \\
&= - \Ad_{g_1^{-1}} \left( T_{(q_0,q_1)} \mathcal{A}_{\mu}  (h^{q_0}(\delta \tau_0),h^{q_1}(\delta \tau_1)) \right).
\end{split}
\]
\end{proof}

\begin{proposition}\label{prop:f_mu}
The force $ \breve{f}_\mu$ coincides with the 1-form $-\breve{\scrA}_\mu$, that is,
\[
\breve{f}_\mu = - \breve{\scrA}_\mu.
\]
\end{proposition}
\begin{proof}
Let us consider $(\delta q_0,\delta q_1) \in T_{(q_0,q_1)} J_\mu^{-1}(\mu)$ and $T_{q_i}\pi_\mu ( \delta q_i) = \delta \tau_i$.
Then,
\begin{eqnarray*}
\breve{\scrA}_{\mu}(\tau_0,\tau_1)(\delta \tau_0,\delta \tau_1) \hspace{-.5pc} & = & \hspace{-.5pc} (\pr_2^* \frakA_\mu - \pr_1^* \frakA_\mu)(q_0,q_1)(\delta q_0,\delta q_1) \\
& = & \hspace{-.5pc} \langle \mu , \frakA(\delta q_1) \rangle - \langle \mu , \frakA(\delta q_0) \rangle \\
& = & \hspace{-.5pc} \langle \mu , \frakA(\delta q_1) - \frakA(\delta q_0) \rangle \\
& \stackrel{\text{Lemma \ref{lemma:adjunta-conexion discreta es diferencia de la conexion principal}}}{=} & \hspace{-.5pc} \langle \mu , - \Ad_{g_{1}^{-1}} \left( T_{(q_0,q_1)} \mathcal{A}_{\mu}  (h_{\mu}^{q_0}(\delta \tau_0),h_{\mu}^{q_1}(\delta \tau_1)) \right) \rangle \\
& = & \hspace{-.5pc} - \langle \Ad_{g^{-1}}^* \mu , T_{(q_0,q_1)} \mathcal{A}_{\mu}  (h_{\mu}^{q_0}(\delta \tau_0),h_{\mu}^{q_1}(\delta \tau_1)) \rangle \\
& = & \hspace{-.5pc} - \langle \mu , T_{(q_0,q_1)} \mathcal{A}_{\mu}  (h_{\mu}^{q_0}(\delta \tau_0),h_{\mu}^{q_1}(\delta \tau_1)) \rangle \\
& \stackrel{\text{Lemma \ref{lemma:fmu en terminos de mu}}}{=} & \hspace{-.5pc} - \breve{f}_\mu(\tau_0,\tau_1)(\delta \tau_0,\delta \tau_1),
\end{eqnarray*}
where we have used that $Ad_{g_{1}^{-1}}^* \mu = \mu$, since $g_{1}^{-1} \in G_\mu$.
\end{proof}

\begin{corollary}\label{cor:f_breve is routh}
The force $\breve{f}_\mu$ is a Routh force. Explicitly,
\[
-d\breve{f}_\mu = \pr_2^* \beta_\mu - \pr_1^* \beta_\mu,
\]
where $\beta_\mu$ is the $2$-form on $Q/G_\mu$ given by dropping $d\frakA_\mu$ to the quotient.
\end{corollary}
\begin{proof}
The only thing that remains is to check that $d\frakA_\mu$ drops to the quotient. Given $q \in Q$, $\xi \in \Lie(G_\mu)$ and $\delta q \in T_qQ$, we have, using Cartan's Magic Formula,
\[
i_{\xi_Q} (d\frakA_\mu) = \calL_{\xi_Q} \frakA_{\mu} - d (i_{\xi_Q} \frakA_{\mu}) = \calL_{\xi_Q} \frakA_{\mu} - d(\langle \mu , \xi \rangle) = \calL_{\xi_Q} \frakA_{\mu} = 0,
\]
where, in the last equality, we have used that $\frakA_\mu$ is $G_\mu$-invariant.

Therefore, $d\frakA_{\mu}$ annihilates vertical vectors and drops to a $2$-form on $Q/G_\mu$.
\end{proof}

\begin{example}
Now we compute the form $\beta_\mu$ for the example we have been analyzing. By \eqref{eq:A_conec}, we have that $\frakA(\delta q) = (0,\delta x - \nu \delta y,0)$ for $\delta q \in T_qQ$. Then,
$$\frakA_\mu(\delta q) = \langle \mu , \frakA(\delta q) \rangle = \mu_2 (\delta x - \nu \delta y).$$
	
As a $1$-form, $\frakA_\mu = \mu_2 (dx - \nu \ dy)$. This implies that $d\frakA_\mu = 0$ and, therefore, $\beta_\mu \equiv 0$.
\end{example}

To complete the comparison, the only thing that remains is checking that both reduced Lagrangians are the same function.

\begin{proposition}
The reduced Lagrangian $\breve{L}_{\mu}$ is the restriction of $L_d$ to the submanifold $J_\mu^{-1}(\mu)$ dropped to the quotient $Q/G_{\mu} \times Q/G_{\mu}$. That is, $\breve{L}_\mu = \widetilde{L}_\mu$.
\end{proposition}
\begin{proof}
Let us recall $\breve{L}_d$ that is defined as $\breve{L}_\mu = \hat{L}_\mu \circ \calY_\mu$. If $(\tau_0,\tau_1) \in (Q/G_\mu) \times (Q/G_\mu)$, then
\[
\begin{split}
\breve{L}_\mu(\tau_0,\tau_1) &= \hat{L}_\mu \left( \calY_\mu(\tau_0,\tau_1) \right) = \hat{L}_\mu \left( \rho(q_0,e) , \tau_1 \right) \\
&= \check{L}_\mu(q_0,e,\tau_1) = L_d(q_0 , \widetilde{F}^{\mu} _1(q_0,e,\tau_1)),
\end{split}
\]
where $q_0 \in \pi_\mu^{-1}(\tau_0)$.

Finally,
\[
\widetilde{F}^{\mu}_1(q_0,e,\tau_1) = l_e^Q \left( \overline{h_{d,\mu}^{q_0}}(\tau_1) \right) = \overline{h_{d,\mu}^{q_0}}(\tau_1) = \pr_2(h_{d,\mu}^{q_0}(\tau_1)) = \pr_2(q_0,q_1),
\]
where $(q_0,q_1) \in \Hor_{\mathcal{A}_{\mu}}=J_\mu^{-1}(\mu) $ and $\pi_\mu(q_1) = \tau_1$. Therefore,
\[
\check{L}_\mu(\tau_0,\tau_1) = L_d(q_0,q_1),
\]
where $(\tau_0,\tau_1)=(\pi_\mu(q_0),\pi_\mu(q_1))$ and $(q_0,q_1) \in \Hor_{\mathcal{A}_{\mu}}$.
\end{proof}

We have, then, the following commutative diagram:
$$\xymatrixcolsep{5pc}\xymatrixrowsep{4pc}\xymatrix{
	& & \rr \\
	J_d^{-1}(\mu) \ar@/^/[urr]^{L_\mu} \ar[d]_-{\widetilde{\pi}_\mu} & & \\
	J_d^{-1}(\mu)/G_\mu \ar[r]^-{\Phi_\mu} & (Q \times \{ e \})/G_\mu \ar@/^/[uur]^-{\hat{L}_\mu} & (Q/G_\mu) \times (Q/G_\mu) \ar[l]_-{\calY_\mu} \ar[uu]_-{\breve{L}_\mu}
}$$	

The previous discussion shows that Theorem \ref{4-teor:red-routh-Leok} is recovered when an abelian symmetry group is considered in Theorem \ref{hor-sym-theor-mu}, so that the latter may be interpreted as a version of discrete Routh reduction for non abelian symmetry groups.

\section{Symplectic structures on the reduced space}

This section begins with the following observation: according to Corollary \ref{cor:f_breve is routh}, the force $\breve{f}_\mu$ obtained after applying a discrete Routh reduction process to a system $(Q,L_d)$ is a Routh force generated by a $2$-form $\beta_\mu$ on $Q/G_\mu$. That is,
\[
-d\breve{f}_\mu = \pr_2^* \beta_\mu - \pr_1^* \beta_\mu.
\]

Hence, we are in the hypotheses of Proposition \ref{regularity}, which says that the $2$-form
\[
\breve{\omega}^+ := (\ff_{\breve{f}_\mu}^+ \breve{L}_\mu)^* (\omega_{Q/G_\mu} - \pi_{Q/G_\mu}^* \beta_\mu)
\]
is preserved by the flow $\bF_{\breve{L}_\mu,\breve{f}_\mu}$ of $(Q/G_\mu,\breve{L}_\mu,\breve{f}_\mu)$, whose existence we will assume\footnote{If the flow exists for the original system, then it also exists for the reduced one.}, where $\ff_{\breve{f}_\mu}^+ \breve{L}_\mu : Q/G_\mu \times Q/G_\mu \lra T^*(Q/G_\mu)$ is the forced discrete Legendre transform given by
\[
\ff_{\breve{f}_\mu}^+ \breve{L}_\mu := (\tau_1 , D_2\breve{L}_\mu(\tau_0,\tau_1) + \breve{f}_\mu^+(\tau_0,\tau_1))
\]
and $\omega_{Q/G_\mu}$ is the canonical symplectic structure on $T^*(Q/G_\mu)$.

We wonder now if $\breve{\omega}^+$ is a symplectic structure, which is equivalent to the regularity of the reduced system, according to Proposition \ref{regularity}. To study this, we will see that $\breve{\omega}^+$ coincides with symplectic structure that arises after applying the Symplectic Reduction Theorem of Marsden--Weinstein, concluding that the reduced system is regular.

We begin noticing that if $L_d$ is regular, we have the symplectic manifold $(Q \times Q, \omega_{L_d})$ and, applying the Symplectic Reduction Theorem (see, for example, Theorem 4.3.1 of \cite{Foundations}), we have the following result:

\begin{proposition}
Let $G_\mu$ be a symmetry group of $(Q,L_d)$ with $L_d$ regular and $G_\mu$-regular and $Q$ a connected manifold. If  $\mu \in J_\mu(Q\times Q)$, where $J_\mu : Q \times Q \lra \frakg_\mu^*$ is the associated discrete momentun map, then there exists a unique symplectic structure $\omega_{L_d,\mu}^r$ on $J_\mu^{-1}(\mu)/G_\mu$ such that
\begin{equation}\label{MW-d}
\widetilde{\pi}_\mu^* \omega_{L_d,\mu}^r = \iota_\mu^* \omega_{L_d},
\end{equation}
where $\iota_\mu : J_\mu^{-1}(\mu) \hookrightarrow Q \times Q$ is the canonical inclusion.
\end{proposition}

\begin{corollary}\label{M-W-reduction}
Let $\alpha_\mu:J_\mu^{-1}(\mu)/G_\mu \lra     (Q/G_\mu) \times (Q/G_\mu)$ be the diffeomorphism defined on Corollary \ref{alpha-mu}. If $\omega_\mu^d:=( \alpha_\mu^{-1})^* \omega_{L_d,\mu}^r$ then $((Q/G_\mu) \times (Q/G_\mu),\omega_\mu^d)$ is a symplectic manifold.
\end{corollary}
\begin{proof}
Since $\alpha_\mu$ is a diffeomorphism with inverse $\delta_\mu$,  $\omega_\mu^d=( \delta_\mu)^* \omega_{L_d,\mu}^r$  is a symplectic structure on $(Q/G_\mu) \times (Q/G_\mu)$.
\end{proof}

We want to compare $\omega_\mu^d$ with the 2-form $\omega_{\breve{f}_\mu^+} = (\ff_{\breve{f}_\mu}^+ \breve{L}_\mu)^* \omega_{Q/G_\mu}$ considered in Section \ref{Sec:forced_strctures}. In order to do so, we consider the following diagram.
$$\xymatrixcolsep{5pc}\xymatrixrowsep{4pc}\xymatrix{
J_\mu^{-1}(\mu) \ar[d]_-{\widetilde{\pi}_\mu} \ar@{^{(}->}[r] & Q \times Q \ar[r]^-{\ff^+L_d} & T^*Q \ar@{-->}[d]^-? \\
J_\mu^{-1}(\mu)/G_\mu \ar[r]_-{\alpha_\mu} & (Q/G_\mu) \times (Q/G_\mu) \ar[r]_-{\ff_{\breve{f}_\mu}^+ \breve{L}_\mu} & T^*(Q/G_\mu)
}$$

The idea is to define a map $\hat{\pi}^{+} : \ff^+ (J_\mu^{-1}(\mu)) \lra T^{*}(Q/G_\mu)$ such that this diagram is commutative; i.e.,  $\hat{\pi}^{+}(\ff^+L_d(q_0,q_1)) = \ff_{\breve{f}_\mu}^+ \breve{L}_\mu(\tau_0,\tau_1)$, where $(q_0,q_1) \in J_\mu^{-1}(\mu)$ and $(\tau_0,\tau_1) = \widetilde{\pi}_\mu(q_0,q_1)$.

Let $\frakA_\mu$ be the $G_\mu$-invariant $1$-form on $Q$ considered in \eqref{3-eq:frakA_mu}. If $[\cdot]$ represent $\cdot$ as a $1$-form on $Q/G_\mu$, we define $\hat{\pi}^+ : \ff^+L_d (J_\mu^{-1}(\mu)) \lra T^*(Q/G_\mu)$ as
$$\hat{\pi}^{+}(q,\alpha) := (\pi_\mu(q),[\alpha - \frakA_\mu(q)]).$$

\begin{proposition}
The map $\hat{\pi}^{+}$
 is well defined and the diagram results commutative.  
\end{proposition}
\begin{proof}
The form $\alpha_q - \frakA_\mu(q)$ defines a $1$-form on $Q/G_\mu$. Indeed, if $\alpha_q \in \ff^+(J_\mu^{-1}(\mu))$, i.e., $\alpha_q = D_2L_d(q_0,q_1)$ with $J_\mu (q_0,q_1) = \mu$ we obtain that
$$\alpha_q(\xi_Q(q)) - \frakA_\mu(q)(\xi_Q(q)) = \langle D_2L_d(q_0,q_1) , \xi_Q(q) \rangle - \langle \mu , \xi \rangle = J_\mu(q_0,q_1)\xi - \langle \mu , \xi \rangle = 0. $$

To check that the diagram commutes, we first recall that, by Proposition \ref{prop:f_mu}, $\breve{f}_\mu = - \breve{\scrA}_\mu$, so that
\begin{equation}\label{eq:fmu breve +}
\breve{f}_\mu^+(\tau_0,\tau_1)(\delta \tau_1) = - \langle \mu , \frakA(\delta q_1) \rangle = - \frakA_\mu(q_1)(\delta q_1).
\end{equation}

Since, $L_d = \breve{L}_\mu \circ \widetilde{\pi}_\mu$, we have that $dL_d = \widetilde{\pi}_\mu^* d\breve{L}_\mu$ and
$$dL_d(q_0,q_1)(0,\delta q_1) = (\widetilde{\pi}_\mu^* (d\breve{L}_\mu))(q_0,q_1)(0,\delta q_1) = d\breve{L}_\mu(\tau_0,\tau_1)(0,\delta \tau_1),$$
implying that $D_2L_d(q_0,q_1)(\delta q_1) = D_2 \breve{L}_\mu(\tau_0,\tau_1)(\delta \tau_1)$. 

Finally, 
\[ D_2L_d(q_0,q_1)(\delta q_1) - \frakA_\mu(q_1)(\delta q_1) = D_2\breve{L}_\mu(\tau_0,\tau_1)(\delta \tau_1) + \breve{f}_\mu^+(\tau_0,\tau_1)(\delta \tau_1). \]
and the diagram is commutative.
\end{proof}

\begin{remark}
Notice that equation \eqref{eq:fmu breve +} seems to imply that $\breve{f}_\mu^+$ does not depend on $\tau_0$. However, this is not necessarily true, since the $q_1$ that appears must satisfy not only the equation $\pi_\mu(q_1) = \tau_1$, but also the condition that $(q_0,q_1)$ belongs to $J_\mu^{-1}(\mu)$.
\end{remark}

Let us now compute $(\hat{\pi}^{+})^* \theta_{Q/G_\mu}$, where $\theta_{Q/G_\mu}$ is the canonical $1$-form on\\ $T^*( {Q/G_\mu})$. Recalling the definition of the canonical $1$-form on cotangent bundles, given $w_{\alpha_q} \in T_{\alpha_q} T^*Q$,
\begin{equation*}
\begin{split}
(\hat{\pi}^{+})^* \theta_{Q/G_\mu}(\alpha_q)(w_{\alpha_q}) %&= \theta_{Q/G_\mu}(\hat{\pi}^{+}(\alpha_q))\left( T\hat{\pi}^{+} ( w_{\alpha_q}) \right) \\
&= \hat{\pi}^{+}(\alpha_q) \left( T\pi_{Q/G_\mu}  \left( T\hat{\pi}^{+}  w_{\alpha_q} \right) \right) \\
%&= \hat{\pi}^{+}(\alpha_q) \left( T(\pi_{Q/G_\mu} \circ \hat{\pi}^{+}) ( w_{\alpha_q}) \right) \\
&= \hat{\pi}^{+}(\alpha_q) \left( T(\pi_\mu \circ \pi_Q) ( w_{\alpha_q}) \right) \\
&=  \left( \alpha_q - \frakA_\mu(q) \right) \left( T\pi_Q  (w_{\alpha_q}) \right) \\
%&= \alpha_q \left( T\pi_Q ( w_{\alpha_q}) \right) - \frakA_\mu(q)\left( T\pi_Q ( w_{\alpha_q}) \right) \\
&=\theta_Q(\alpha_q)(w_{\alpha_q}) - (\pi_Q^* \frakA_\mu)(w_{\alpha_q}).
\end{split}
\end{equation*}

Hence, $(\hat{\pi}^{+})^* \theta_{Q/G_\mu} = \theta_Q - \pi_Q^* \frakA_\mu$ and
\begin{equation}\label{eq:pihat pullback de omega QGmu}
(\hat{\pi}^{+})^* \omega_{Q/G_\mu} = \omega_Q - \pi_Q^* (-d\frakA_\mu).
\end{equation}

This shows that $(\ff_{\breve{f}_\mu}^+ \breve{L}_\mu)^* \omega_{Q/G_\mu}$ is not exactly $\omega_\mu^d$, but we can modify the canonical structure $\omega_{Q/G_\mu}$ to get another structure whose pullback by the forced discrete Legendre transform $\ff_{\breve{f}_\mu}^+ \breve{L}_\mu$ does agree with $\omega_\mu^{d}$.

Noticing that
\[
\pi_Q^* (- d\frakA_\mu) = -\pi_Q^* \pi_\mu^* \beta_\mu = -(\hat{\pi}^+)^* \left( \pi_{Q/G_\mu}^* \beta_\mu \right),
\]
returning to \eqref{eq:pihat pullback de omega QGmu},
\[
(\hat{\pi}^{+})^* \omega_{Q/G_\mu} = \omega_Q + (\hat{\pi}^+)^* \left( \pi_{Q/G_\mu}^* \beta_\mu \right).
\]

Therefore,
\[
(\hat{\pi}^{+})^* \left( \omega_{Q/G_\mu} - \pi_{Q/G_\mu}^* \beta_\mu \right) = \omega_Q.
\]

\begin{proposition}
Let $\omega_{Q/G_\mu}$ be the canonical symplectic structure on $T^*(Q/G_\mu)$ and $\pi_{Q/G_\mu} : T^*(Q/G_\mu) \lra Q/G_\mu$  the cotangent projection. Then,
$$(\ff_{\breve{f}_\mu}^+ \breve{L}_\mu)^* \left( \omega_{Q/G_\mu} - \pi_{Q/G_\mu}^* \beta_\mu \right) = \omega_\mu^d.$$
\end{proposition}

We summarize all these results in a single statement.

\begin{theorem}\label{thm:symplecticity and discrete routh reduction}
Let $(Q/G_\mu,\breve{L}_\mu,\breve{f}_\mu)$ be the FDMS obtained after applying the discrete Routh reduction procedure to a regular discrete Lagrangian system $(Q,L_d)$. Then, the system is regular and the $2$-form
\[
\breve{\omega}^+ := (\ff_{\breve{f}_\mu}^+ \breve{L}_\mu)^* (\omega_{Q/G_\mu} - \pi_{Q/G_\mu}^* \beta_\mu)
\]
is a symplectic structure preserved by the flow of the system, where $\beta_\mu$ is the $2$-form on $Q/G_\mu$ obtained by dropping $d\frakA_\mu$ to the quotient.
\end{theorem}
\begin{proof}
The only thing that remained to be checked is the regularity of the system, which is a consequence of the fact that, by the previous proposition, $\breve{\omega}^+ = \omega_\mu^d$. Therefore, $\breve{\omega}^+$ is a symplectic structure and we can apply Proposition \ref{regularity}.
\end{proof}

\begin{remark}
Just as we did in Section \ref{Sec:forced_strctures}, the work we have done with the forced discrete Legendre transform $\ff_{\breve{f}_\mu}^+ \breve{L}_\mu$ can also be applied to $\ff_{\breve{f}_\mu}^- \breve{L}_\mu$.
\end{remark}

\begin{remark}
As in the continuous case, we observe that if $ \omega_{Q/G_\mu}$ is modified by adding the term $\pi_{Q/G_\mu}^* \beta_\mu$, that we may call \emph{discrete magnetic term} (see for example \cite{Foundations}), its pullback by the forced discrete Legendre transform defines a preserved symplectic structure on $(Q/G_\mu) \times (Q/G_\mu)$.
\end{remark}

\section{Numerical experiments with a central potential in the plane}

Here we briefly illustrate how some of the reduction techniques
developed in the previous sections may be applied to a simple
mechanical system and compare the results to those obtained by
applying a generic numerical integrator ---order $4$ Runge--Kutta
method, \texttt{RK4}.

We consider a central potential mechanical system in $\R^2$; such
systems are symmetric under the rotation action of $S^1$ on $\R^2$. A
convenient description is given using polar coordinates or, better,
using the covering map $\R^1 \longrightarrow S^1$. The resulting system,
has configuration manifold $Q := \R_{>0}\times\R$ and the Lagrangian
is
\begin{equation}\label{eq:lagrangian_PC-def-short}
	L(r,\eta,\dot{r},\dot{\eta}) :=
	\frac{m}{2} (\dot{r}^2 + r^2 \dot{\eta}^2) - V(r).
\end{equation}
The $S^1$ symmetry becomes an $\R$ symmetry acting by translation in
the second coordinate of $Q$. Notice that the quotient map
$\pi:Q\longrightarrow Q/\R$ can be identified with the projection
$\pr_1:Q\longrightarrow \R_{>0}$.

As mentioned in Remark \ref{remark:discrete vs cont routh reduction}, a reduction procedure that can be applied
to this system is described
in \cite{ar:marsden_ratiu_scheurle:2000:reduction_theory_and_the_lagrange-routh_equations}. This
procedure requires the use of two constructions derived from the
kinetic part of $L$: the mechanical connection and the locked inertia
tensor. In our setting, the connection $1$-form is
$\frakA(r,\eta) = d\eta$; whereas, for $(r,\eta)\in Q$, the locked
inertia tensor $\mathfrak{I}_{(r,\eta)}:\frakg \longrightarrow \frakg^*$ is
$\mathfrak{I}_{(r,\eta)}(\xi) = m r^2\xi\, 1^*$. Given
$\mu\, 1^* \in\frakg^* = \R^*$ we define
$\frakA_\mu := \mu \frakA = \mu d\eta$; we also define the amended potential
$V_\mu(r,\eta) := V(r,\eta) + \frac{1}{2} \mu
(\mathfrak{I}_{(r,\eta)})^{-1}(\mu) = V(r) + \frac{\mu^2}{2mr^2}$,
that naturally defines a function on $\R_{>0}$, denoted by
$\mathfrak{V}_\mu(r)$. According to Theorem III.9
in \cite{ar:marsden_ratiu_scheurle:2000:reduction_theory_and_the_lagrange-routh_equations},
the reduced system has configuration space $Q/\R\simeq \R_{>0}$ and
its dynamics is determined by a variational principle involving the
reduced Routhian $\mathfrak{R}^\mu:T(Q/G)\longrightarrow \R$ given by
\begin{equation}\label{eq:reduced_routhian-central_potential-def-short}
	\mathfrak{R}^\mu(r,\dot{r}) := \frac{1}{2} \norm{(r,\dot{r})}^2_{Q/\R}-\mathfrak{V}_\mu(r) = \frac{m}{2}\dot{r}^2 - V(r) - \frac{\mu^2}{2m
		r^2}
\end{equation}
and a certain $2$-form $\beta_\mu$ on $\R_{>0}$ which, for dimensional
reasons, must vanish. Then, the reduced system is a Lagrangian system,
whose Lagrangian function is $\mathfrak{R}^\mu$. Hence, the equation
of motion ---the Euler--Lagrange Equation for $\mathfrak{R}^\mu$--- is
\begin{equation}\label{eq:reduced_equation_of_motion-central_potential-short}
	\begin{split}
		0 = \frac{d}{dt} \frac{\partial \mathfrak{R}^\mu(r(t),\dot{r}(t))}{\partial \dot{r}} - \frac{\partial \mathfrak{R}^\mu(r(t),\dot{r}(t))}{\partial r} =
		m\ddot{r}(t) - \frac{\mu^2}{m r(t)^3}+V'(r(t)).
	\end{split}
\end{equation}
Also as a consequence of being a Lagrangian system with Lagrangian
$\mathfrak{R}^\mu$, its flow preserves the ``energy'' function
\begin{equation}\label{eq:routhian_energy-def-short}
	\mathfrak{E}^\mu(r,\dot{r}) := \frac{m}{2}\dot{r}^2 + V'(r) +
	\frac{\mu^2}{2m r^2},
\end{equation}
as well as the symplectic structure
$\omega_\mu := \left(\ff \mathfrak{R}^\mu\right)^*(\omega_{\R_{>0}}) =
m dr\wedge d\dot{r}$, where $\omega_{\R_{>0}}$ is the canonical
symplectic structure on $T^*\R_{>0}$.

For most potential functions $V:\R_{>0}\longrightarrow \R$,
equation~\eqref{eq:reduced_equation_of_motion-central_potential-short}
can be solved numerically. In what follows, when we talk about the
exact solution
of~\eqref{eq:reduced_equation_of_motion-central_potential-short} we
actually mean the numerical solution provided by \texttt{Mathematica}'s
adaptive and high-order integrator \texttt{NDSolve}.

For the numerical examples that follow we take
\begin{equation}
	\label{eq:the_sextic_potential-def}
	V(r) := \alpha r^2(r^2-\beta)^2,
\end{equation}
for $\alpha := 0.1$ and $\beta := 2$, as well as $m := 1$. We use
initial conditions $(x_0,y_0) := (0,0.2)$ and
$(\dot{x}_0,\dot{y}_0) := (0.57,0.01)$, that lead to
$(r_0,\eta_0) := (0.2,1.5708)$ and
$(\dot{r}_0.\dot{\eta}_0) := (0.01,-2.85)$, as well as
$\mu := -0.114$. Last, we use a time-step $h := 0.2$.

In order to compare with the variational integrator considered later,
we integrate
numerically~\eqref{eq:reduced_equation_of_motion-central_potential-short}
using the generic integrator
\texttt{RK4}. Figure~\ref{fig:RK4_error_t_0_100-1.4.4-short} shows the
evolution of the (global) error in $r$ and
Figure~\ref{fig:MPR_RK4_exact_evolution_energy_t_0_100-1.4.4-short}
shows the evolution of $\mathfrak{E}^\mu$ (small brown squares), both
using \texttt{RK4}.
\begin{figure}[htb]
	\centering \subfigure[Error in $r$ (\texttt{RK4})\label{fig:RK4_error_t_0_100-1.4.4-short}]
	{\includegraphics[scale=.4]{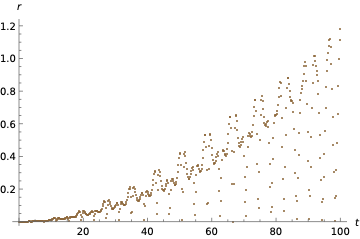}}
	\goodgap \subfigure[Error in $r$ (\texttt{MP})\label{fig:MPR_error_t_0_100-1.4.4-short}]
	{\includegraphics[scale=.4]{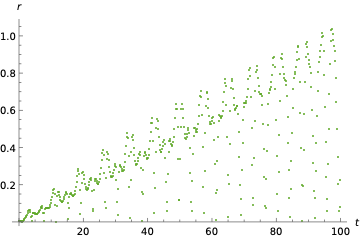}}
	\caption{Error in $r$ in the evolution of the reduced systems
		computed using \texttt{RK4} (initial conditions $r_0 := 0.2$ and
		$\dot{r}_0 := 0.01$) and using \texttt{MP} (initial conditions
		$r_0 := 0.2$ and $r_1 := 0.201$). In both cases, $h := 0.2$.}
	\label{fig:RK4_evolution_t_0_100-short}
\end{figure}

\begin{figure}[htb]
	\centering
	\includegraphics[scale=.5]{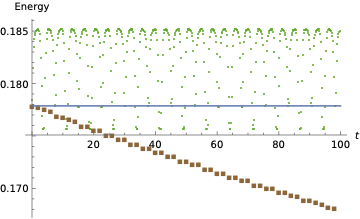}
	\caption{Comparison of the evolution of $\mathfrak{E}^\mu$: exact
		value (straight blue line), \texttt{RK4} (brown squares) and
		\texttt{MP} (small green dots).  In all cases, $h := 0.2$.}
	\label{fig:MPR_RK4_exact_evolution_energy_t_0_100-1.4.4-short}
\end{figure}

Going back to the mechanical system $(Q,L)$, we use a midpoint
discretization to produce the discrete Lagrangian
\begin{equation}\label{eq:L_d^MP-def}
	\begin{split}
		L_d^{MP}(r_0,\eta_0,r_1,\eta_1):=& h L\left(\frac{r_1-r_0}{h},\frac{\eta_1-\eta_0}{h},\frac{r_1+r_0}{2},\frac{\eta_1+\eta_0}{2}\right) \\=&                                                                                  h\bigg( \frac{m}{2} \left( \left(\frac{r_1-r_0}{h}\right)^2+\left(\frac{r_1+r_0}{2}\right)^2\left(\frac{\eta_1-\eta_0}{h}\right)^2\right)
		\\&\phantom{h\bigg(}- V\left(\frac{r_1+r_0}{2}\right)\bigg).
	\end{split}
\end{equation}

\begin{remark}
	This discretization is, by no means, unique. The choice of the
	discrete Lagrangian will affect the approximation order achieved by
	the
	integrator. (See \cite{ar:fernandez_graiffZurita_grillo:2021:error_analysis_of_forced_discrete_mechanical_systems}
	for the local error analysis.)
\end{remark}

As $(Q,L_d^{MP})$ is $\R$-symmetric for the same action as above, the
procedure described in Section \ref{sec:routh}, constructs\footnote{The
	construction described in Section \ref{sec:routh} uses a principal connection on
	$\pi:Q\longrightarrow Q/\R$. The connection that we use has horizontal
	lift $\HLc{(r,\eta)}(\del_r) := \del_r \in T_{(r,\eta)}Q$.} the
reduced forced discrete mechanical system
$(\R_{>0},\breve{L}_d^{MP},\breve{f}_d^{MP})$ where
\begin{equation*}
	\begin{split}
		\breve{L}_d^{MP}(r_0,r_1) =& h\bigg( \frac{m}{2} \left( \left(\frac{r_1-r_0}{h}\right)^2+ \left(\frac{\mu}{m}\right)^2\left(\frac{2}{r_0+r_1}\right)^2\right) - V\left(\frac{r_1+r_0}{2}\right)\bigg),\\
		\breve{f}_d^{MP}(r_0,r_1) =& \frac{8h\mu^2}{m}\frac{d r_0+d r_1}{(r_0+r_1)^3}.
	\end{split}
\end{equation*}

\begin{remark}\label{rem:modified_lagrangian-MP-short}
	It is easy to check that $\breve{f}_d^{MP} = d\gamma_d^{MP}$ for
	\begin{equation*}
		\gamma_d^{MP}(r_0,r_1) := -\frac{4 h \mu^2}{m (r_0+r_1)^2}.
	\end{equation*}
	Thus, as mentioned in Remark \ref{remark:absorbing the force}, the discrete forced mechanical system\\
	$(\R_{>0},\breve{L}_d^{MP},\breve{f}_d^{MP})$ is equivalent to the
	\emph{unforced} discrete mechanical system\\
	$(\R_{>0},\breve{L}_d^{MP}+\gamma_d^{MP})$. An immediate consequence
	---consistent with Theorem \ref{thm:symplecticity and discrete routh reduction}--- is that the flow of the system preserves the
	symplectic structure obtained pulling back the canonical structure
	$\omega_{\R_{>0}}$ by $\ff (\breve{L}_d^{MP}+\gamma_d^{MP})$.
\end{remark}

\begin{remark}
	Comparison of the modified Lagrangian $\breve{L}_d^{MP}+\gamma_d$
	introduced in Remark~\ref{rem:modified_lagrangian-MP-short} with the
	reduced
	Routhian~\eqref{eq:reduced_routhian-central_potential-def-short}
	shows that
	\begin{equation*}
		(\breve{L}_d^{MP}+\gamma_d)(r_0,r_1) = h \mathfrak{R}^\mu\left(\frac{r_1+r_0}{2},\frac{r_1-r_0}{h}\right).
	\end{equation*}
	Thus, in the case of central potentials in $\R^2$, the reduced
	forced discrete mechanical system constructed out of $L_d^{MP}$ is
	equivalent to the Mid Point integrator for the (continuous) reduced
	system associated with $L$. It is interesting to explore other more
	general settings where the reduction and discretization procedures
	commute.
\end{remark}

The numerical integrator \texttt{MP} consists of solving the discrete
equations of motion \eqref{forcedELe} for
$(\R_{>0},\breve{L}_d^{MP},\breve{f}_d^{MP})$. Using the potential $V$
given by~\eqref{eq:the_sextic_potential-def} and the same initial data
as with the \texttt{RK4} integrator,
Figure~\ref{fig:MPR_error_t_0_100-1.4.4-short} shows the (global)
error in $r$ and
Figure~\ref{fig:MPR_RK4_exact_evolution_energy_t_0_100-1.4.4-short}
shows the evolution of $\mathfrak{E}^\mu$ (small green dots), both
using \texttt{MP}.

Comparing Figures~\ref{fig:RK4_error_t_0_100-1.4.4-short}
and~\ref{fig:MPR_error_t_0_100-1.4.4-short} we observe that both
\texttt{RK4} and \texttt{MP} exhibit similar (global) error profiles
in the current setting. On the other hand,
Figure~\ref{fig:MPR_RK4_exact_evolution_energy_t_0_100-1.4.4-short}
shows that \texttt{RK4} and \texttt{MP} fail to preserve
$\mathfrak{E}^\mu$ in quite different ways: whereas the value of
$\mathfrak{E}^\mu$ decreases steadily with \texttt{RK4}, it fluctuates
around the (constant) exact value with \texttt{MP}, a fact that has
been verified for other time-steps and initial conditions. Notice that
this difference is very relevant for the long term simulation of
systems where (approximate) conservation of $\mathfrak{E}^\mu$ is
essential to the qualitative reproduction of the system's behavior.

%%%%%%%%%%%%%%%%%%%%%%%%%%%%%%%%%%%%%%%%%%%%%%%%%%%%%%%%%%%%%%%%%%%%%%%%%%

\section*{Acknowledgments}
This document is the result of research partially supported by grants
from the Universidad Nacional de Cuyo [code 06/C009-T1 and code 06/80020240100069UN],
Universidad Nacional de La Plata [code X915] and CONICET.

%%%%%%%%%%%%%%%%%%%%%%%%%%%%%%%%%%%%%%%%%%%%%%%%%%%%%%%%%%%%%%%%%%%%%%%%

\printbibliography

\end{document}